%% file: main.tex
\documentclass[12pt]{amsart}
\usepackage[T1]{fontenc}
\usepackage{amsmath,amssymb,amsthm}
\usepackage{mathtools}
\usepackage{tikz-cd}
\usepackage{cleveref}
\usepackage{fullpage}
\usepackage{algorithm}
\usepackage{algpseudocode}
\usepackage{comment}
\usepackage{booktabs}
\usepackage{breqn}
\usepackage{listings}
\usepackage{xcolor}

\usepackage[backend=biber,style=alphabetic]{biblatex}
\lstdefinelanguage{Macaulay2}{
  keywords={ZZ, QQ, ideal, diff, transpose, sub, rank, matrix, map, ker, image, sum, apply, intersect, saturate, degrees, trim, gens, positions, degree, return, and, first, ring, coefficientRing, inducedMap, loadPackage, generators, numColumns, submatrix, target, for, in, list, toSequence, endl, mingens, source, select, toList, print, scan},
  morecomment=[l]{--}, 
  morestring=[b]",     
  sensitive=true       
}

\lstdefinestyle{terminal}{
    basicstyle=\scriptsize\ttfamily,
    keywords={},        
    breaklines=true,
    frame=single,
    numbers=none
}

\newcommand{\C}{\mathbb{C}}

\newcommand{\N}{\mathbb{N}}
\newcommand{\Z}{\mathbb{Z}}

\newcommand{\PP}{\mathbb{P}}
\renewcommand{\O}{\mathcal{O}}
\newcommand{\K}{\mathcal{K}}

\newcommand{\E}{\mathcal{E}}
\newcommand{\st}{\ : \ }
\newcommand{\sym}{\operatorname{Sym}}

\newcommand{\sing}{\operatorname{Sing}}
\newcommand{\pic}{\operatorname{Pic}}
\newcommand{\bs}{\operatorname{Bs}}

\newcommand{\im}{\operatorname{Image}}
\newcommand{\coker}{\operatorname{coker}}
\newcommand{\Gr}{\operatorname{Gr}^\bullet}
\newcommand{\Proj}{\operatorname{Proj}}
\newcommand{\GGL}{Green--Griffiths--Lang}

\theoremstyle{plain}
\newtheorem{thm}{Theorem}[section]
\newtheorem{lemma}[thm]{Lemma}
\newtheorem{prop}[thm]{Proposition}
\newtheorem{cor}[thm]{Corollary}

\theoremstyle{definition}
\newtheorem{defn}[thm]{Definition}
\newtheorem{example}{Example}

\newtheorem*{ggl}{Green--Griffiths--Lang Conjecture}

\title{Computing jet differentials and the Green--Griffiths--Lang Conjecture for complements of smooth plane curves}
\author{Ryan Contreras, Joseph Cummings, Eric Riedl, and Jaziel Torres}
\date{}

\begin{document}
\begin{abstract}
    We study the {\GGL} Conjecture for complements of smooth plane curves.
    We develop an effective method for computing a family of negatively twisted invariant logarithmic 2-jet differentials. 
    By realizing the first logarithmic jet space as a hypersurface in $\PP^2 \times \PP^2$, we encode these jet differentials in a finitely generated bigraded module that can be computed explicitly.
    We use this description to give a computational criterion for the {\GGL} Conjecture and verify it for several families of smooth plane curves.
    In examples with sufficiently many independent jet differentials, we determine the exceptional locus explicitly.
\end{abstract}
\maketitle

\section{Introduction}

The {\GGL} Conjecture is a major conjecture governing holomorphic maps from $\C$ to varieties of general type. We state its logarithmic form, which is the focus of this paper.
\begin{ggl}
    Let $X$ be a smooth projective variety over $\C$ and $D$ a simple normal crossing divisor such that $K_X + D$ is big.
    Then there is a proper subvariety $Z\subset X$ containing all entire curves of $X\setminus D$.
    That is, for any nonconstant holomorphic map $\gamma:\C\to X\setminus D$, $\gamma(\C)\subseteq Z$.
\end{ggl}

Because it predicts fundamental properties of an important class of varieties, the {\GGL} Conjecture has been the subject of intense study over the years.
For instance, just for the case $X = \PP^2$ we have the works \cite{green77, babets, dethloffSchumacherWong, el2003logarithmic, HouHuynhMerkerXie2026} studying {\GGL} and the works \cite{chen2004algebraichyperbolicity, PacienzaRousseau, CorvajaZannier2013, chenRiedlYeong} studying an (easier) algebraic analog of the {\GGL} Conjecture.
Despite all of this important work, however, the conjecture remains open for complements of plane curves. In this paper, we focus on this case, with $X = \PP^2$ and $D = C$ a smooth plane curve.

Over the past several decades, tools have been developed to restrict entire curves in $X \setminus D$. 
Extending some ideas of Green and Griffiths \cite{green1980two}, Demailly \cite{demailly1997algebraic} introduced varieties $\pi: X_k \to X$ for each $k$ that roughly speaking track points of $X$ together with possible $k$th order derivatives of a curve on $X$. 
Each $X_k$ comes with a line bundle $\O_{X_k}(\xi_k)$ satisfying the following basic property: for $a > 0$, any section of $\O_{X_k}(m\xi_k-aH)$ must vanish on the lift of any entire curve from $X \setminus D$.  
Thus, these line bundles provide ways to restrict entire curves on $X \setminus D$.
The direct image sheaf $\pi_*\O_{X_k}(m\xi_k)$ is the bundle $E_{k,m}\Omega_X(\log D)$ of invariant jet differential operators of order $k$ and total degree $m$.


Our computational focus is the family of line bundles
\[
    \O_{X_2}(\xi_2 + b\xi_1 - aH), \qquad a,b > 0,
\]
whose coefficient of the second tautological class $\xi_2$ is equal to one.
Their sections are particularly suitable for explicit calculation and for studying the geometry of their zero loci. 
Since the remaining part of their divisor class is pulled back from $X_1$, the first step is to determine the effective and nef cones of $X_1$, characterizing when a line bundle on $\O_{X_1}(r\xi_1 + tH)$ has a nonzero section.

Using results of El Goul \cite{el2003logarithmic}, we then show that the {\GGL} Conjecture for the complement of a given smooth plane curve can, in many cases, be verified using only the spaces $H^0(X_2, \O_{X_2}(\xi_2 + b\xi_1 - aH))$.
This yields the following criterion.

\begin{thm}\label{test theorem}
    Let $C \subset \PP^2$ be a smooth curve of degree $d$.
    Then $\PP^2 \setminus C$ satisfies {\GGL} if any of the following are satisfied: \begin{enumerate}
        \item[(I)] $d\geq 15$ and $H^0(X_2,\O_{X_2}(\xi_2 + b\xi_1 - aH)) = 0$ for all $a, b > 0$.
        \item[(II)] $d\geq 11$, $a, b > 0$, $H^0(X_2,\O_{X_2}(\xi_2 + b\xi_1 - aH))$ has at least one section whose vanishing locus is irreducible, and 
        \[
            \frac{a}{b + 1} < \frac{4d^2-51d+90}{12(d-3)}.
        \]
        \item[(III)] $d \geq 4$ and there are two sections $\sigma_1 \in H^0(X_2,\O_{X_2}(\xi_2+b_1\xi_1 - a_1H))$ and $\sigma_2 \in H^0(X_2,\O_{X_2}(\xi_2+b_2\xi_1-a_2H))$ with $a_1,a_2,b_1,b_2 > 0$, and with $\sigma_1$ and $\sigma_2$ not both vanishing along the same divisor in $X_2$.
    \end{enumerate}
\end{thm}

Furthermore, we are able to embed $X_1 = \PP T_{\PP^2}(-\log C)$ into $\PP^2 \times \PP^2$ and use this explicit representation of $X_1$ to effectively compute $H^0(X_2,\O_{X_2}(\xi_2 + b\xi_1 - aH))$ for all $a,b > 0$.


\begin{thm}\label{thm:introAlgorithm}
    Given a smooth plane curve $C$, the spaces
    \[
        H^0(X_2,\O_{X_2}(\xi_2+b\xi_1-aH)), \qquad a,b>0,
    \]
    can be computed effectively from a finitely generated bigraded module.
\end{thm}


Theorem 1.2 is made precise in \Cref{V1 as maps of modules} together with \Cref{lemma:sections of V1}.
Using the algorithm from Theorem \ref{thm:introAlgorithm}, we are able to apply Theorem \ref{test theorem} to many different examples, exhibiting a wide range of behavior with respect to our test. 
For instance, we have the following.

\begin{cor}
    The following curves have no sections of $H^0(X_2,\O_{X_2}(\xi_2 + b\xi_1 - aH))$ for all $a,b > 0$. 
    Hence their complements satisfy {\GGL} by \Cref{test theorem}(I): 
    \begin{itemize}
        \item $x_0^{15} + x_1^{15} + x_2^{15} + x_0^5 x_1^5 x_2^5$
        \item $x_0^{15} + x_1^{15} + x_2^{15} + x_0x_1x_2(x_0 + x_1 + x_2)^{12}$
        \item $x_0^{17} + 2x_1^{17} + 3x_2^{17} + 13x_0^6x_1^5x_2^6$
        \item $x_0^{14}x_1 + x_0x_2^{14} + x_1^{14}x_2 + x_0^5 x_1^5 x_2^5$
    \end{itemize}
    For each of the following curves, there exists $a,b > 0 $ satisfying $\frac{a}{b + 1} < \frac{4d^2-51d+90}{12(d-3)}$ such that $H^0(X_2,\O_{X_2}(\xi_2 + b\xi_1 - aH))$ admits a section with irreducible vanishing locus.
    Hence, their complements satisfy {\GGL} by \Cref{test theorem}(II): 
    \begin{itemize}
        \item $x_0^{13}+x_0^6x_1^7+x_1^{13}+x_1x_2^{12}$
        \item $x_0^{13}+x_0^6x_1^7+x_1^{13}+x_1x_2^{12}+x_2^{13}$
        \item $x_0^{18} + x_1^{18} + x_2^{18} + x_0^9x_1^9 + x_0^{12}x_1^4x_2^2$
    \end{itemize}
\end{cor}

For curves satisfying criterion (III) in \Cref{test theorem}, our explicit computations allow us to not only confirm {\GGL}, but also fully characterize the exceptional locus.
\begin{cor}\label{cor: exceptional locus description}
    For each of the following curves, there are least two sections of $H^0(X_2, \O_{X_2}(\xi_2+b\xi_1-aH))$, possibly for different $a,b > 0$, whose common vanishing locus has codimension 2 in $X_2$.
    Hence, their complements satisfy {\GGL} by \Cref{test theorem}(III) and their exceptional loci are described in the table below. 
    \begin{center}
    \renewcommand{\arraystretch}{1.4} 
    \begin{tabular}{@{} l p{8.5cm} @{}}
        \toprule
        \textup{\textbf{Polynomial defining}} $C\subset\PP^2$ & \textup{\textbf{Exceptional locus}} $\operatorname{Exc}(\PP^2\setminus C)$ \\
        \midrule
        
        $x_0^{2n} + x_0^n x_1^n + x_1^{2n} + x_1 x_2^{2n-1}$ \ ($n \ge 7$) & 
        The union of the $2n+1$ lines $\lambda x_0 - x_1 = 0$, \newline 
        where $\lambda = 0$ or $\lambda^{2n} + \lambda^n + 1 = 0$. \\
        
        $x_0^{2n+1}  + x_0^{n}x_1^{n+1} + x_1^{2n+1} + x_1x_2^{2n}$ ($n \geq 7$) & 
        The union of the $2n+2$ lines $\lambda x_0 - x_1 = 0$, \newline 
        where $\lambda = 0$ or $\lambda^{2n+1}+\lambda^{n+1} + 1 = 0$. \\
        
        $x_0^{2n} + x_0^nx_1^n + x_1^{2n} + x_1x_2^{2n-1} + x_2^{2n}$, ($n\geq 7$) & 
        The union of the $2n$ lines $\lambda x_0-x_1 = 0$, \newline
        where $\lambda^{2n}+\lambda^n + 1 = 0$. \\
        
        $x_0^{2n} + x_1^{2n} + x_2^{2n} + c x_0^nx_1^n$, ($n \geq 7$, $c\notin\{0,\pm 2\}$) & 
        The union of the $2n$ lines $\lambda x_0-x_1 = 0$, \newline
        where $\lambda^{2n}+c\lambda^n + 1 = 0$ \\
        
        $x_0^{14}+2x_0^7x_1^7+x_1^{14}-x_1^{13}x_2+x_2^{14}$ & 
        $\emptyset$ \ (The complement is hyperbolic) \\
        \bottomrule
    \end{tabular}
    \end{center}
\end{cor}

Finally, there are curves where none of the conditions (I)--(III) of \Cref{test theorem} hold, so that our test is inconclusive.

\begin{cor}
    For each of the following curves, $h^0(X_2,\O(\xi_2+b\xi_1 - aH)) = 1$ for a single pair $(a,b)$ satisfying $\frac{a}{b + 1} \geq \frac{4d^2-51d+90}{12(d-3)}$, and $h^0(X_2,\O(\xi_2+b\xi_1 - aH)) = 0$ for all other $(a,b)$ with $a,b > 0$.
    Thus, \Cref{test theorem} cannot be applied: 
    \begin{itemize}
        \item $x_0^d + x_1^d + x_2^d$ ($7 \leq d \leq 20$)
        \item $x_0^d + x_2^d - x_1^{d-1}x_2$ ($8 \leq d \leq 20$)
        \item $x_0^d + x_1^d + x_2^d + x_0^{d-1}x_1$ ($8 \leq d \leq 20$)
        \item $x_0^d + x_1(x_0^{d-1} + x_1^{d-1} + x_2^{d-1})$ ($9 \leq d \leq 20$)
        \item $x_0^d + x_1^d + x_2^{d-2}(x_0^2 + x_1^2 + x_2^2)$ ($9 \leq d \leq 20$)
        \item $x_0^{d-1}x_1 + x_1^{d-1}x_2 + x_2^{d-1}x_0$ ($10 \leq d \leq 20$)
        \item $x_0^d + x_1^d + x_2^{d-2}(x_0^2 + x_0x_1 + x_0x_2 + x_1^2 + x_1x_2 + x_2^2)$ ($d = 15,16)$
    \end{itemize}
\end{cor}
We remark that by \cite[Theorem 1]{toda1971functional}, the complement of the Fermat curve $x_0^d + x_1^d + x_2^d$ satisfies {\GGL} for $d\geq 7$.
On the other hand, the curve $x_0^d + x_1^d + x_2^{d-2}(x_0^2 + x_1^2 + x_2^2)$ is proved to be hyperbolic for all $d\geq 9$ and $d\not\equiv 2 \mod 4$ in the recent paper \cite[Theorem A]{nguyen2026hyperbolicity}.

We believe that our computational tools will be useful in testing {\GGL} for other curves in $\PP^2$, and hope that this framework will move us closer to understanding {\GGL} for complements of all smooth plane curves of large degree.

Recent work of Hou, Huynh, Merker, and Xie \cite{HouHuynhMerkerXie2026} develops a different computer-assisted method that works for the full spaces $E_{2,m}\Omega_{\PP^2}(\log C)\otimes\O_{\PP^2}(-a)$.
For fixed $(m,a)$, they write every possible section on an affine chart in terms of finitely many unknown coefficient polynomials. 
Requiring this expression to remain holomorphic when rewritten in the other affine charts imposes linear equations on those coefficients.
The desired space of sections is then recovered as the kernel of this linear system.
Thus, their method computes both the dimension of $H^0(\PP^2, E_{2,m}\Omega_{\PP^2}(\log C)\otimes\O_{\PP^2}(-a))$ and explicit representatives of all its sections. 
Their implementation is used primarily to establish vanishing results that lead to the hyperbolicity of the complement of a generic plane curve of degree at least 12.

Our approach is different.
For any given smooth plane curve we compute the spaces of sections
\[
    H^0(X_2,\O_{X_2}(\xi_2+b\xi_1 - aH)) \qquad \text{ for all } a,b > 0.
\]
By pushing the tautological line bundle $\O_{X_2}(\xi_2)$ down to $X_1\subset \PP^2\times\PP^2$, we collect these spaces of global sections into a single bigraded module over the bihomogeneous coordinate ring of $X_1$.
In this description, each section is represented globally by a four-tuple of bihomogeneous polynomials.
This particularly simple and explicit representation makes their zero loci accessible to direct algebraic computation, allowing us to obtain the explicit results in \Cref{cor: exceptional locus description}.


The outline of the paper is as follows.
In \Cref{sec:jet-diff}, we review invariant jet differentials and establish the criterion on the dimension of the stable base locus used throughout the paper.
In \Cref{sec:1-jets}, we realize $X_1$ as a hypersurface in $\PP^2\times\PP^2$ and describe its effective and nef cones. 
In \Cref{sec:2-jets}, we revisit results of El Goul regarding 2-jet differentials in the setting of complements of smooth plane curves and prove the {\GGL} criterion of Theorem 1.1. 
In \Cref{sec:Explicit Computations}, we give the bigraded-module description underlying Theorem 1.2 (Corollary 5.4) and a Macaulay2 implementation. 
Finally, the computations establishing Corollaries 1.3, 1.4, and 1.5 are carried out in Sections 6.1–6.2, 6.3, and 6.4, respectively.
The Macaulay2 files used to verify the computations in \Cref{sec:examples} are available at \url{https://github.com/jazieltorres/jet-differentials-plane-curves}.

\textbf{Acknowledgments:} We gratefully acknowledge conversations with Izzet Coskun, Jon Hauenstein, Tony V\'{a}rilly Alvarado, Noa Hunter and Bruno de Oliveira. 
This research was supported in part by the Notre Dame Center for Research Computing through access to its high-performance computing clusters and data storage infrastructure.
Eric Riedl was supported by NSF CAREER grant DMS-1945944 and Simons Foundation grants 00011850 and 00013673.

\textbf{AI Disclosure:} While searching for examples of curves to apply our hyperbolicity tests to, we used AI to make suggestions and then modified the suggestions until we found curves with the desired behavior. We also used ChatGPT to catch grammar and mathematical errors as we finalized the draft. The paper was written by the authors, who are accountable for the correctness of the results.

\section{Invariant Jet Differentials}\label{sec:jet-diff}
Let $X$ be a smooth projective variety of dimension $n$ with a simple normal crossing divisor $D$.
The logarithmic tangent sheaf $T_{X}(-\log D)$ of $X$ along $D$ is the subsheaf of the tangent sheaf $T_{X}$ consisting of holomorphic vector fields which are tangent to $D$.
As such, it fits in the exact sequence
\[\begin{tikzcd}
    0 \rar & T_{X}(-\log D) \rar & T_{X} \rar{\rho} & \O_D(D).
\end{tikzcd}
\]
The map $\rho$ acts on vector fields as restriction to $D$ followed by projection to the normal direction. 
If $U = X\setminus D$, then $T_{X}(-\log D)|_U \cong T_U$.
On the other hand, if $z_1,\dots, z_n$ are local coordinates for $X$ around $p\in D$, near which $D = \{z_1\cdots z_r = 0\}$, then the stalk $T_{X}(-\log D)_p$ is the free $\O_{X,p}$-module generated by 
\[
    z_1\frac{\partial}{\partial z_1},\dots, z_r\frac{\partial}{\partial z_r}, \frac{\partial}{\partial z_{r+1}}, \dots, \frac{\partial}{\partial z_n}.
\]
As $D$ is assumed simple normal crossing, $T_{X}(-\log D)$ is locally free of rank $n = \dim X$.
Note that $\rho$ is surjective if $D$ is smooth. 

Starting with the pair $\big(X,T_{X}(-\log D)\big)$ one can construct a tower of projective bundles
\[\begin{tikzcd}
   \cdots\rar & X_k \rar["\pi_k"] & \cdots\rar & X_2 \rar["\pi_2"] & X_1 \rar["\pi_1"] & X_0 = X,
\end{tikzcd}
\]
with the property that each $X_k$ is endowed with a locally free sheaf $V_k$ and $X_{k+1} = \PP V_k$.
Starting with $X_0 = X$ and $V_0 = T_{X}(-\log D)$, the next level is constructed as follows.
Set $X_1 = \PP V_0$ and let $\pi_1:X_1\to X$ be the natural projection.
The projective bundle $X_1$ is equipped with a tautological line bundle $\O_{X_1}(-1) \subset \pi_1^*V_0$.
Define $V_1$ as the subbundle of $T_{X_1}(-\log \pi_1^*D)$ that is the preimage of the tautological line bundle.
As such, $V_1$ fits in the exact sequence
\begin{equation}\label{seq:V1}
\begin{tikzcd}
    0 \rar & T_{X_1/X_0}\rar & V_1\rar{d\pi_1} & \O_{X_1}(-1)\rar & 0.
\end{tikzcd}
\end{equation}
In other words, the fiber of $V_1$ over the point $(x,[v])\in X_1$ is the set of logarithmic tangent vectors in the fiber of $T_{X_1}(-\log \pi_1^*D)$ over $(x,[v])$ whose projection to $X$ is contained in the line in $T_{X}(-\log D)_x$ spanned by $v$.
Given $X_1$ and $V_1$, we set $X_2 = \PP V_1$ and the construction continues inductively defining $V_k$ via a sequence analogous to \eqref{seq:V1}.

An \textit{entire curve} in $X\setminus D$ is a nonconstant holomorphic map $\gamma:\C\to X\setminus D$, and the \textit{exceptional locus} of $X\setminus D$ is the Zariski closure of the union of the images of all entire curves. 
Let $\pi:X_k\to X$ be the projection.
An entire curve has a unique and well-defined lift $\gamma_{[k]}:\C\to X_k\setminus \pi^*D$.
The direct image sheaf $\pi_*\O_{X_k}(m)$ is the bundle $E_{k,m}\Omega_X(\log D)$ of invariant jet differential operators of order $k$ and total degree $m$.
This is the set of germs of polynomial differential operators on germs of holomorphic curves, where the action is invariant under reparametrization. 
For the case of $1$-jets, $E_{1,m}\Omega_X(\log D) = \sym^m\Omega_X(\log D)$, and for $2$-jets there is a filtration 
\begin{equation}\label{eq:Gr_filtration}
    \Gr E_{2,m}\Omega_X(\log D) = \bigoplus_{0 \leq i \leq \lfloor m/3 \rfloor} \sym^{m-3i}\Omega_X(\log D)\otimes\overline{K}_X^{\otimes i},
\end{equation}
where $\overline{K}_X = K_X\otimes\O(D)$ is the logarithmic canonical sheaf of $(X,D)$. 
For a comprehensive analysis on invariant jet differential operators, see \cite{demailly1997algebraic,dethloff2001logarithmic}.
Their usefulness in the study of the exceptional locus of $X\setminus D$ becomes clear from the following.
\begin{thm}\label{thm:ggl_locus}
    Fix positive integers $k,m$ and an ample line bundle $A$ on $X$.
    Assume 
    \[
        H^0(X_k,\O_{X_k}(m)\otimes\pi^*A^{-1}) \cong H^0(X,E_{k,m}\Omega_X(\log D)\otimes A^{-1})
    \]
    has nonzero sections $\sigma_1,\dots,\sigma_N$ and denote by $Z\subset X_k$ their base locus.
    Then any entire curve $\gamma:\C\to X\setminus D$ must satisfy $\gamma_{[k]}(\C)\subseteq Z$.
\end{thm}

Fix an ample line bundle $A$ on $X$ and let $B_k\subset X_k$ be the stable base locus
\[
    B_k = \bigcap_{m\in\N} \bs\big|\O_{X_k}(m)\otimes\pi^*A^{-1}\big|.
\]
If the projection of some $B_k$ down to $X$ is a proper closed set, then $X$ satisfies the {\GGL} Conjecture.
One approach to the {\GGL} Conjecture is thus to find enough algebraically independent sections in $H^0(X_k,\O_{X_k}(m)\otimes\pi^*A^{-1})$ in order to reduce as much as possible the dimension of the nonvertical components of $B_k$.
Here \textit{vertical} means two things: components that are completely contained in a fiber of $X_k\to X$ or the divisor $\Gamma_k = \PP(T_{X_{k-1}/X_{k-2}})\subset X_k$.
The divisor $\Gamma_k$ records the projectivized tangent directions that are vertical with respect to the projection $X_{k-1}\to X_{k-2}$ and, as such, cannot contain lifts of entire curves from $X$.

Reducing the dimension of the nonvertical components of $B_k$ to zero or one automatically shows {\GGL} for the logarithmic space $(X,D)$.
Moreover, in the case where $(X,D)$ is a logarithmic surface of general type, reducing the nonvertical components of $B_k$ to dimension 2 is enough to conclude {\GGL} even when these components dominate $(X,D)$.
The argument proceeds by associating to such a two-dimensional component of $B_k$ a foliation on a desingularization and then combining El Goul's logarithmic degeneracy theorem with Jouanolou's theorem on algebraic leaves.

\begin{thm}\label{thm:McQuillan}
    If $(X,D)$ is a logarithmic surface of general type and if the nonvertical components of $B_k$ have dimension at most 2, then $X\setminus D$ satisfies {\GGL}, i.e., there is a proper subvariety $Z\subset X$ containing all entire curves on $X\setminus D$. 
\end{thm}
\begin{proof}
    By considering each nonvertical component of $B_k$ one at a time, we may assume that $B_k$ is irreducible.
    If $\pi:B_k\to X$ is not dominant, we are done.
    Otherwise, let $Y$ be the projection of $B_k$ down to $X_1$ so that $Y$ is a hypersurface in $X_1$ that dominates $X$.
    As in \cite[Lemma 3.1]{demailly2000hyperbolicity}, $T_Y\cap V_1$ induces a (possibly singular) foliation $\mathcal{F}$ on a desingularization $\widetilde{Y}$ of $Y$.
    Let $\widetilde{D}$ be the reduced preimage of $D$ in $\widetilde{Y}$.
    Note that because $(X,D)$ is of general type, $(\widetilde{Y},\widetilde{D})$ is as well.
    For any entire curve $\gamma:\C\to X\setminus D$, the lift $\gamma_{[1]}$ lies inside $Y$ and unless $\gamma_{[1]}(\C)\subset\sing(Y)$, it lifts to $\widetilde{Y}$ and is tangent to $\mathcal{F}$.
    By \cite[Theorem 2.4.2]{el2003logarithmic}, every lifted entire curve is algebraically degenerate; hence its Zariski closure is an invariant algebraic curve.
    Entire curves contained in the singular locus already project to a proper algebraic subset of $X$.
    
    If there are only finitely many invariant algebraic curves, the proof is finished.
    Otherwise, \cite[Proposition on p.240]{jouanolou1978hypersurfaces} (see also \cite[Theorem 2.2]{demailly2000hyperbolicity}) implies $\mathcal{F}$ is the relative tangent foliation of a meromorphic fibration onto a curve $C$.
    After replacing $\widetilde{Y}$ by a further resolution, we may assume that $p:\widetilde{Y}\to C$ is a fibration and let $F$ be the general fiber.
    Still denoting by $\widetilde{D}$ the reduced preimage of $D$ in this new $\widetilde{Y}$, $K_{\widetilde{Y}} + \widetilde{D}$ is big and consequently it intersects positively with $F$.
    By adjunction,
    \[
        (K_{\widetilde{Y}} + \widetilde{D})\cdot F = 2g(F)-2 + \#(\widetilde{D}\cap F) > 0
    \]
    making $F\setminus\widetilde{D}$ hyperbolic.
    It follows that only finitely many exceptional fibers can contain images of lifted entire curves.
    The projections to $X$ of these finitely many fibers and of the singular locus of $Y$ therefore form a proper algebraic subset containing all the entire curves in $X\setminus D$.
\end{proof}


\subsection{The Case of Smooth Plane Curves}
In the present paper we apply this approach to the case of $X = \PP^2$ and $D = C$ a smooth plane curve.
Moreover, we will only deal with jet differentials of order $1$ and order $2$.
In this specific situation, we summarize and set notation.
Since $T_{\PP^2}(-\log C)$ has rank 2, $X_1 = \PP T_{\PP^2}(-\log C)$ is a $\PP^1$-bundle over $\PP^2$ and $X_2$ is a $\PP^1$-bundle over $X_1$, with respective projections $\pi_1:X_1\to\PP^2$ and $\pi_2:X_2\to X_1$.
We let $\pi:X_2\to \PP^2$ be the composition $\pi_1\circ\pi_2$, i.e., the projection down to $\PP^2$. 
\[\begin{tikzcd}
    X_2 \rar{\pi_2}\arrow[bend right=30,swap]{rr}{\pi} & X_1 \rar{\pi_1} & \PP^2
\end{tikzcd}
\]
Let $H$ be the hyperplane class of $\PP^2$, $\xi_1 = c_1(\O_{X_1}(1))$, and $\xi_2 = c_1(\O_{X_2}(1))$.
By \Cref{thm:ggl_locus}, sections of 
\[
    H^0\big(X_1,\O(m\xi_1-aH)\big)\qquad\text{or}\qquad H^0\big(X_2,\O(m\xi_2-aH)\big)
\]
must vanish on lifts of entire curves on $\PP^2\setminus C$ and thus we wish to study (the nonvertical components of) the stable base loci
\[
    B_1 = \bigcap_{m\in\N}\bs\big| m\xi_1-aH \big| \qquad\text{and}\qquad B_2 = \bigcap_{m\in\N} \bs\big| m\xi_2-aH \big|.
\]
In the next section, we prove that $B_1 = X_1$, which implies that 1-jets do not constrain entire curves.
The following sections are then devoted to identifying conditions under which we can deduce that the nonvertical components of $B_2$ are of dimension at most 2, and to giving examples in which we can compute these components explicitly. In the cases where we have an explicit multifoliation, we study its leaves. 

\section{Geometry of the First Jet Space}\label{sec:1-jets}
Fix a smooth plane curve $C$ of degree $d$.
In this section, we will show that $X_1 = \PP T_{\PP^2}(-\log C)$ can be embedded as a hypersurface in $\PP^2 \times \PP^2$.
This will allow us to compute the effective cone of $X_1$ and show that any divisor $m\xi_1-aH$ with $m,a > 0$ is not effective. 

\begin{prop}
    Let $C\subset\PP^2$ be a smooth curve defined by the homogeneous polynomial $f\in\C[x_0,x_1,x_2]$.
    Then, $X_1 = \PP T_{\PP^2}(-\log C)$ is a hypersurface in $\PP^2\times\PP^2$ defined by the bihomogeneous polynomial
    \[
        g = a_0 \frac{\partial f}{\partial x_0} + a_1 \frac{\partial f}{\partial x_1} + a_2 \frac{\partial f}{\partial x_2} \quad\in\quad \C[x_0,x_1,x_2,a_0,a_1,a_2].
    \]
\end{prop}
\begin{proof}
    We start by writing the Euler sequence on $\PP^2$, the defining sequence of $T_{\PP^2}(-\log C)$, and the ideal sheaf sequence of $C$ twisted by $\O_{\PP^2}(C) = \O_{\PP^2}(d)$ together in one diagram.
    \begin{equation}\label{diagram1}
    \begin{tikzcd}
    	&& 0\dar & 0\dar  & \\
    	&& {\O_{\PP^2}}\dar & {\O_{\PP^2}}\dar &  \\
    	&& {\O_{\PP^2}(1)^3}\dar & {\O_{\PP^2}(d)}\dar & \\
    	0\rar & {T_{\PP^2}(-\log C)}\rar & {T_{\PP^2}}\dar\rar & {\O_C(C)}\dar\rar & 0 \\
    	&& 0 & 0 &
    \end{tikzcd}
    \end{equation}
    The Jacobian matrix $\nabla f = \begin{bmatrix} \frac{\partial f}{\partial x_0} & \frac{\partial f}{\partial x_1} & \frac{\partial f}{\partial x_2} \end{bmatrix}$ gives a map $\nabla f : \O_{\PP^2}(1)^3\to\O_{\PP^2}(d)$ which is surjective for $C$ smooth and commutes with the diagram \eqref{diagram1}, resulting in the following diagram.
    \[\begin{tikzcd}
    	&& 0\dar & 0\dar & \\
    	&& {\O_{\PP^2}}\dar\rar[equal] & {\O_{\PP^2}}\dar & \\
    	0\rar & \ker(\nabla f)\rar & {\O_{\PP^2}(1)^3}\rar["\nabla f"]\dar & {\O_{\PP^2}(d)}\dar\rar & 0 \\
    	0\rar & {T_{\PP^2}(-\log C)}\rar & {T_{\PP^2}}\rar\dar & {\O_C(C)}\rar\dar & 0 \\
    	&& 0 & 0 &
    \end{tikzcd}\]
    By the Snake Lemma, we have an isomorphism $\ker\left(\nabla f \right) \cong T_{\PP^2}(-\log C)$ and hence the fundamental exact sequence for the logarithmic tangent sheaf on $\PP^2$:
    \begin{equation}\label{seq:euler_log}
    \begin{tikzcd}
        0 \rar & {T_{\PP^2}(-\log C)} \rar & {\O_{\PP^2}(1)^3} \rar{\nabla f} & {\O_{\PP^2}(d)} \rar & 0.
    \end{tikzcd}
    \end{equation}
    Projectivizing we get a closed immersion
    \[
        X_1 = \PP T_{\PP^2}(-\log C) \hookrightarrow \PP\O_{\PP^2}(1)^3 \cong \PP^2\times\PP^2
    \]
    realizing $X_1$ as a hypersurface in $\PP^2\times\PP^2$.
    Setting coordinates $x_0,x_1,x_2$ for the first $\PP^2$ factor (the base space) and $a_0,a_1,a_2$ coordinates for the second $\PP^2$ factor (the fibers of $\PP(\O_{\PP^2}(1)^3)$), the hypersurface $X_1\subset \PP^2\times\PP^2$ is defined by the vanishing of the bihomogeneous polynomial $g = \sum a_i\frac{\partial f}{\partial x_i}$.
\end{proof}

Points of $X_1\subset\PP^2\times\PP^2$ are pairs $(x,a)$ such that the vector $a = (a_0,a_1,a_2)$ picks out the one-dimensional subspace inside the fiber of $\O_{\PP^2}(1)^{3}$ over $x$ that defines the point $\PP(T_{\PP^2}(-\log C)_x)$.
Furthermore, the pullbacks of the hyperplane classes via the inclusion $i: X_1 \hookrightarrow \PP^2 \times \PP^2$ are given by
\begin{equation}\label{eq:pullback classes}
    c_1( i^*\O_{\PP^2 \times \PP^2}(1, 0)) = H, \qquad c_1 (i^*\O_{\PP^2 \times \PP^2}(0, 1)) = \xi_1 + H.
\end{equation}
As a result we obtain the following, which is a more precise version of a special case of a result due to Schneider and others \cite{Schneider1992}.

\begin{thm}\label{thm:eff_cone}
    Let $C\subset\PP^2$ be a smooth curve and let $X_1 = \PP T_{\PP^2}(-\log C)$ with projection $\pi_1:X_1\to\PP^2$.
    If $\xi_1 = c_1(\O_{X_1}(1))$ and $H = c_1(\pi_1^*\O_{\PP^2}(1))$, then
    \[
        \overline{\textup{Eff}}(X_1) = \textup{Nef}(X_1) = \{r\xi_1 + tH \st t \geq r \geq 0\}.
    \]
\end{thm}
\begin{proof}
    The Picard group of $X_1$ has rank 2 and is generated by the classes $\xi_1$ and $H$, $\pic(X_1) = \Z\xi_1\oplus\Z H$.
    Writing $X_1$ as a hypersurface in $\PP^2\times\PP^2$, $H$ and $\xi_1+H$ are the pullback of $\O_{\PP^2}(1)$ from each factor.
    As such, $H$ and $\xi_1+H$ are effective moving classes with $H^3 = 0$ and $(\xi_1+H)^3 = 0$; thus $H$ and $\xi_1+H$ are extremal effective divisors. 
    Moreover, they are both nef divisors, and since the nef cone is contained in the effective cone,
    \[
        \overline{\textup{Eff}}(X_1) = \textup{Nef}(X_1) = \{r\xi_1 + tH \st t \geq r \geq 0\}.
    \]
\end{proof}
As an immediate consequence of \Cref{thm:eff_cone}, $H^0(X_1,\O(m\xi_1-aH)) = 0$ whenever $m > 0$ and $a\geq 0$, recovering, with a different proof, the result in \cite[Lemma 1.4.1]{el2003logarithmic} (see also \cite[Corollary 3]{diverio2009existence}).
Therefore it is impossible for 1-jets to constrain the exceptional locus of $\PP^2\setminus C$.

\section{2-Jet Differentials and the Green-Griffiths-Lang Test}\label{sec:2-jets}
For the rest of the paper we keep as a standing assumption that $C\subset\PP^2$ is a smooth curve of degree $d\geq 4$, so that the logarithmic pair $(\PP^2,C)$ is of general type.

The base locus of 2-jets on logarithmic surfaces was investigated by El Goul \cite{el2003logarithmic}. 
Notably, he established a logarithmic analog of McQuillan's theorem regarding the algebraic degeneracy of entire curves on foliated surfaces of general type \cite[Theorem 2.4.2]{el2003logarithmic} (see also \cite[Theorem 3.12]{campana2020orbifold}).

In this section, we collect several key results from \cite{el2003logarithmic}, adapting the proofs specifically to the setting of complements of plane curves.
The objective is to identify conditions that guarantee the existence of sufficient sections in $H^0(X_2,\O(m\xi_2-aH))$ to ensure that the nonvertical components of the base locus $B_2$ have codimension at least 2 in $X_2$.

Recall that $X_2 = \PP V_1$, where $V_1$ fits into the exact sequence
\begin{equation}\label{ses:def V_1}
\begin{tikzcd}
    0 \rar & T_{X_1/\PP^2}\rar & V_1\rar{d\pi_1} & \O_{X_1}(-\xi_1)\rar & 0.
\end{tikzcd}
\end{equation}
We have a canonical injection $\O_{X_2}(-\xi_2)\hookrightarrow\pi_2^*V_1$ and thus, by pulling the above exact sequence back to $X_2$, we obtain a morphism of line bundles 
\[
    \O_{X_2}(-\xi_2)\xrightarrow{\pi_2^*(d\pi_1)} \pi_2^*\O_{X_1}(-\xi_1).
\]
The zero locus of the section associated to this morphism is the effective divisor $\Gamma_2 := \PP(T_{X_1/\PP^2}) \subset X_2$.
It follows that $\O_{X_2}(\Gamma_2)\cong\O_{X_2}(\xi_2)\otimes\pi_2^*\O_{X_1}(-\xi_1)$, thus $\Gamma_2$ has class $[\Gamma_2] = \xi_2 - \xi_1$.
The lift to $X_2$ of any entire curve from $\PP^2$ cannot be contained in $\Gamma_2$; any entire curve contained in $\Gamma_2$ must be tangent to the fibers of $X_1\to\PP^2$ and so it projects down as a constant map $\C\to\PP^2$.

In the rest of the paper, we make use of the following logarithmic analog of \cite[Lemma 3.3]{demailly2000hyperbolicity} about properties of the line bundles $\mathcal{O}_{X_2}(c\xi_2 + b\xi_1)$.
For the convenience of the reader we provide proofs which follow the same strategy as in \cite[Lemma 3.3]{demailly2000hyperbolicity}. 

\begin{lemma}[cf. {\cite[Lemma 3.3]{demailly2000hyperbolicity}}]\label{lemma:Demailly3.3}
With respect to the projection $\pi:X_2 \to \PP^2$, the line bundle $\mathcal{O}_{X_2}(c\xi_2 + b\xi_1)$ is:
\begin{enumerate}
    \item[(a)] relatively effective (resp., big) if and only if $b + c \geq 0$ and $c \geq 0$ (resp., $b + c > 0$ and $c > 0$), and
    \item[(b)] relatively nef (resp., ample) if and only if $b \geq 2c \geq 0$ (resp., $b  > 2c  > 0$).
\end{enumerate}
Moreover, the following properties hold:
\begin{enumerate}
    \item[(c)] For $m = b + c \geq 0$, there is an injection 
    \[
        \pi_*\mathcal{O}_{X_2}(c\xi_2 + b\xi_1) \hookrightarrow E_{2,m}\Omega_{\PP^2}(\log C),
    \]
    and the injection is an isomorphism if $b - 2c \leq 0$. 
    \item[(d)] Let $Z\subset X_2$ be an irreducible divisor such that $Z\neq \Gamma_2$. 
    Then, in $\pic(X_2)$ we have $Z\sim c\xi_2 + b\xi_1 + aH$ where $b\geq 2c \geq 0$.
\end{enumerate}
\end{lemma}
\begin{proof}
    Let $F\cong\PP^1$ be a fiber of $\pi_1:X_1\to\PP^2$.
    Restricting to $F$ we obtain the short exact sequence
    \[\begin{tikzcd}
        0 \rar & \O_{\PP^1}(2) \rar & V_1\big|_F \rar & \O_{\PP^1}(-1) \rar & 0
    \end{tikzcd}\]
    Since $\text{Ext}^1(\O_{\PP^1}(-1),\O_{\PP^1}(2)) = 0$, the sequence splits writing $V_1\big|_F\cong \O_{\PP^1}(2)\oplus\O_{\PP^1}(-1)$.
    Hence, the fibers of $\pi:X_2\to\PP^2$ are Hirzebruch surfaces 
    \[
        S = \PP(\O_{\PP^1}(2)\oplus\O_{\PP^1}(-1)) \cong \PP(\O_{\PP^1}\oplus\O_{\PP^1}(-3)).
    \]
    Let $\zeta = \xi_2\big|_S$ and $\eta = \xi_1\big|_S$.
    It follows that $\zeta = c_1(\O_S(1))$ and $\eta$ is the class of a fiber, i.e., the pullback of $\O_{\PP^1}(1)$ under the projection $S\to\PP^1$.
    We have the intersection products $\zeta^2 = -1$, $\zeta\cdot \eta = 1$, and $\eta^2 = 0$. 
    Since $[\Gamma_2\big|_S] = \zeta - \eta$ and $(\zeta-\eta)^2 = -3$, $\Gamma_2\big|_S$ is the unique irreducible curve in $S$ with self-intersection $-3$.

    The line bundle $\mathcal{O}_{X_2}(c\xi_2 + b\xi_1)$ is relatively effective/nef if and only if its restriction to the fiber $S$ is.
    It is a standard result that the effective cone of $S$ is 
    \[
        \overline{\textup{Eff}}(S) = \{r(\zeta-\eta) + t\eta \st r\geq 0 \text{ and } t\geq 0\}.
    \]
    By setting $c = r$ and $b = t-r$, we see that the condition $r,t\geq 0$ translates as $c\geq 0$ and $b+c \geq 0$. 
    The big cone is obtained by changing the inequalities to strict inequalities.
    This shows (a).

    For (b), $(c\xi_2 + b\xi_1)\big|_S = c\zeta + b\eta$ is nef if and only if it pairs non-negatively with both $(\zeta-\eta)$ and $\eta$.
    Hence
    \[
        (\zeta-\eta)\cdot(c\zeta + b\eta) = b-2c \geq 0
    \]
    and
    \[
        \eta\cdot(c\zeta + b\eta) = c \geq 0.
    \]
    The ample cone is obtained by changing the inequalities to strict inequalities.

    To prove (c), note that 
    \[
        \O_{X_2}(c\xi_2 + b\xi_1)\otimes\O_{X_2}(b\Gamma_2)\cong \O_{X_2}(m\xi_2).
    \]
    Hence, we get the inclusion $\O_{X_2}(c\xi_2 + b\xi_1)\hookrightarrow\O_{X_2}(m\xi_2)$ if $b\geq 0$, and the reverse inclusion if $b < 0$.
    Restricting to the fiber $S$ of $\pi:X_2\to\PP^2$ and pushing forward to the intermediate fiber,
    \begin{equation}\label{eq:decomp1}
        (\pi_2)_*\big(\O_{X_2}(c\xi_2 + b\xi_1)\big|_S\big) \cong \sym^c (\O_{\PP^1}(-2)\oplus\O_{\PP^1}(1))\otimes\O_{\PP^1}(b) \cong \bigoplus_{i=0}^c \O_{\PP^1}(b+c-3i)
    \end{equation}
    and similarly,
    \begin{equation}\label{eq:decomp2}
        (\pi_2)_*\big(\O_{X_2}(m\xi_2)\big|_S\big) \cong \sym^m (\O_{\PP^1}(-2)\oplus\O_{\PP^1}(1)) \cong \bigoplus_{i=0}^m \O_{\PP^1}(m-3i).
    \end{equation}
    Since $m = b+c$, these two decompositions differ by the factors $\O_{\PP^1}(b+c-3i)$ for $i$ in the range $m < i \leq c$ if $b < 0$, and in the range $c < i \leq m$ if $b \geq 0$.

    If $b\geq 0$, $\O_{X_2}(c\xi_2 + b\xi_1)\hookrightarrow\O_{X_2}(m\xi_2)$ and so pushing forward to $\PP^2$ we get the injection
    \[
        \pi_*\mathcal{O}_{X_2}(c\xi_2 + b\xi_1) \hookrightarrow E_{2,m}\Omega_{\PP^2}(\log C).
    \]
    In this case, $m \geq c$ and assuming $b-2c\leq 0$, then all the factors $\bigoplus_{i=c+1}^m\O_{\PP^1}(b+c-3i) = \bigoplus_{i=c+1}^m\O_{\PP^1}(m-3i)$ in \eqref{eq:decomp2} have negative degree, showing that 
    \[
        H^0\big(\PP^1, (\pi_2)_*(\O_{X_2}(c\xi_2 + b\xi_1)\big|_S)\big)
        \cong
        H^0\big(\PP^1, (\pi_2)_*(\O_{X_2}(m\xi_2)\big|_S)\big),
    \]
    and consequently
    \[
        H^0\big(S,\O_{X_2}(c\xi_2 + b\xi_1)\big|_S\big) \cong H^0\big(S,\O_{X_2}(m\xi_2)\big|_S\big).
    \]
    By Grauert's Theorem, the natural inclusion is an isomorphism on every fiber and hence an isomorphism.
    
    On the other hand, if $b < 0$, we get the reverse inclusion
    \[
        \O_{X_2}(m\xi_2)\hookrightarrow\O_{X_2}(c\xi_2 + b\xi_1).
    \]
    In this case, $m < c$ and so all the factors $\bigoplus_{i=m+1}^c\O_{\PP^1}(b+c-3i) = \bigoplus_{i=m+1}^c\O_{\PP^1}(m-3i)$ in \eqref{eq:decomp1} have negative degree.
    As above we conclude that $\pi_*\mathcal{O}_{X_2}(c\xi_2 + b\xi_1)\cong E_{2,m}\Omega_{\PP^2}(\log C)$.
    This finishes the proof of (c).

    For (d), let $Z\subset X_2$ be an irreducible divisor with $[Z]\sim c\xi_2 + b\xi_1 + aH$.
    If $Z$ does not dominate $\PP^2$, then $b = c = 0$.
    Otherwise, if $Z\neq \Gamma_2$, then the restriction $Z\vert_S$ to a generic fiber $S$ of $\pi:X_2\to\PP^2$ must be a curve in $S$ not containing $\Gamma_2\vert_S$ as a component.
    It follows that the intersection between $Z\vert_S$ and $\Gamma_2\vert_S$ must be nonnegative, i.e., 
    \[
        (c\zeta + b\eta)\cdot(\zeta-\eta) = b-2c\geq 0.
    \]
    Since $Z\vert_S$ is effective, $c\geq 0$ by part (a), completing the proof.
\end{proof}

\begin{prop}[{\cite[Theorem 1.2.1]{el2003logarithmic}}]\label{prop:xi2_big}
    Let $(\PP^2,C)$ be a logarithmic pair with $C$ smooth of degree $d\geq 11$.
    Then $\xi_2 = c_1(\O_{X_2}(1))$ is big.
\end{prop}
\begin{proof}
From the inequality (13.16) in \cite[p.352]{demailly1997algebraic},
\begin{equation*}
    h^0(\PP^2, E_{2,m}\Omega_{\PP^2}(\log C)\otimes\O_{\PP^2}(-1)) \geq \frac{m^4}{648}(13c_1^2 - 9c_2) - O(m^3),
\end{equation*}
where $c_1 = (3-d)H$ and $c_2 = d^2-3d +3$ are the Chern classes of the logarithmic tangent sheaf $T_{\PP^2}(-\log C)$.
Since $h^0(X_2,\O(m\xi_2-H)) = h^0(\PP^2, E_{2,m}\Omega_{\PP^2}(\log C)\otimes\O_{\PP^2}(-1))$ and $X_2$ is a 4-fold, it follows that $\xi_2$ is big whenever $13c_1^2 - 9c_2 > 0$.
That is,
\[
    13c_1^2 - 9c_2 = 4d^2 - 51d + 90 > 0,
\]
which holds for $d \geq 11$.
\end{proof}

\begin{prop}[{\cite[Lemma 1.3.2]{el2003logarithmic}}]\label{prop:ZcupGamma}
    Let $(\PP^2,C)$ be a logarithmic pair with $C$ smooth of degree $d\geq 11$.
    There are positive integers $m$, $a$ and a section $\sigma\in H_0(X_2, \O_{X_2}(m\xi_2 - aH))$ such that $\{\sigma = 0\} =  Z \cup \Gamma_2$,
    where $Z$ is an irreducible divisor. 
\end{prop}
\begin{proof}
    Fix $n$ large enough so that $\O_{X_2}(n\xi_2 - H)$ has a global section (cf. \Cref{prop:xi2_big}) and let $l[\Gamma_2] + \sum k_i[Z_i]$ be the associated effective divisor with $Z_i\subset X_2$ reduced and irreducible, $k_i > 0$, and $Z_i\neq \Gamma_2$.
    Write
    \[
        [Z_i] = c_i\xi_2 + b_i\xi_1 + a_i H.
    \]
    From \Cref{lemma:Demailly3.3}(d), $b_i \geq 2c_i\geq 0$.
    In particular, $b_i + c_i \geq 0$.
    If $b_i + c_i = 0$, then $b_i = c_i = 0$ and $a_i > 0$ as $Z_i$ is effective.
    By re-indexing, we may assume $Z_0$ is an irreducible component where $a_0/(b_0 + c_0)$ is the minimum of all the ratios $a_i /(b_i + c_i)$ with $b_i + c_i > 0$.
    We must have $a_0 < 0$ because $\sum k_ia_i = -1$ and $k_i > 0$.
    Set $a = -a_0$ and $m = c_0 + b_0$.
    Hence, 
    \[
        [Z_0] + b_0[\Gamma_2] = m\xi_2 - aH
    \]
    is effective and there is a section $\sigma\in H_0(X_2, \O_{X_2}(m\xi_2 - aH))$.
    It follows that $\{\sigma = 0\} = Z_0 \cup \Gamma_2$.
\end{proof}

By \Cref{prop:ZcupGamma}, we have a section that restricts the lifts of entire curves to $X_2$ if $d \geq 11$. 
In the next few results we will search for a numerical condition on these sections that guarantees that the dimension of the components of $B_2$ distinct from $\Gamma_2$ is at most 2. 

\begin{prop}[{\cite[Theorem 1.3.3]{el2003logarithmic}}]\label{prop:restriction_is_big}
    Let $(\PP^2,C)$ be a logarithmic pair with $C$ smooth of degree $d \geq 11$.
    Assume $\sigma\in H^0\big(X_2,\O_{X_2}(m\xi_2-aH)\big)$ is a section with $\{\sigma = 0\} = Z\cup\Gamma_2$, where $Z$ is irreducible with $[Z] = c\xi_2 + b\xi_1 - aH$.
    If 
    \begin{equation}\label{ineq:negative_section}
        \frac{a}{b + c} < \frac{4d^2-51d+90}{12(d-3)},
    \end{equation}
    then the restricted divisor $\xi_2\big|_Z$ is big.
    In particular, $\PP^2\setminus C  $ satisfies {\GGL}.
\end{prop}
\begin{proof}
    Given that $Z$ is irreducible, $b \geq 2c \geq 0$. 
    Moreover, $b$ and $c$ cannot be both zero as $[Z]$ is effective. 
    We will first show that $(\xi_2+2\xi_1)^3\cdot[Z] > 0$.
    Consider the table of intersection products
    \[
        \xi_2^4 = -c_1^2 + 5c_2 \qquad 
        \xi_2^3\cdot\xi_1 = c_1^2 - 3c_2 \qquad 
        \xi_2^2\cdot\xi_1^2 = c_2 \qquad
        \xi_2\cdot\xi_1^3 = c_1^2 - c_2 \qquad
        \xi_1^4 = 0
    \]
    \[
        \xi_2^3\cdot H = 0 \qquad \xi_2^2\xi_1 \cdot H = 0 \qquad \xi_2\xi_1^2\cdot H = -c_1 \cdot H \qquad \xi_1^3\cdot H = 0
    \]
    where $c_1 = -(d-3)H$ and $c_2 = d^2 - 3d + 3$ are the Chern classes of $T_{\PP^2}(-\log C)$.
    From it, 
    \begin{align*}
        (\xi_2+2\xi_1)^3\cdot[Z] &= (b + c)(13c_1^2 - 9c_2) + 12ac_1\cdot H \\
        &= (b + c)\bigg(4d^2 - 51d + 90 - 12(d-3)\frac{a}{b + c}\bigg) > 0,
    \end{align*}
    where the inequality follows from the assumption \eqref{ineq:negative_section}.

    Choose some sufficiently small rational number $\epsilon > 0$ so that the above inequality holds after replacing $(\xi_2+2\xi_1)$ with $(\xi_2+(2+\epsilon)\xi_1)$, i.e.,  $(\xi_2+(2+\epsilon)\xi_1)^3\cdot[Z] > 0$.
    Let 
    \[
        L = \O_{X_2}(n\xi_2+(2+\epsilon)n\xi_1)
    \]
    with $n$ large and divisible enough.
    If we can show that $L\vert_Z$ is big, it follows that $\xi_2\vert_Z$ is big as well, since $(3+\epsilon)n\xi_2 = L + (2+\epsilon)n\Gamma_2$.
    By the asymptotic version of Riemann-Roch, to show that $L\vert_{Z}$ is big, it suffices to show that $h^2(Z,L\vert_Z)$ is $o(n^3)$ as $n\to\infty$. 
    We show the stronger statement $h^2(Z,L\vert_Z) = 0$, and the proof of this fact follows as in the proof of \cite[Theorem 1.3.3]{el2003logarithmic}.
    For the convenience of the reader, we provide the main steps.

    From the exact sequence
    \[\begin{tikzcd}
        0 \rar & \O_{X_2}(-Z)\otimes L \rar & L \rar & \O_Z\otimes L \rar & 0
    \end{tikzcd}\]
    and its associated long exact sequence in cohomology, 
    \[
        h^2(Z,L\vert_Z) \leq h^2(X_2,L) + h^3(X_2, \O_{X_2}(-Z)\otimes L).
    \]
    Thus, it suffices to show the vanishing of $h^2(X_2,L)$ and $h^3(X_2, \O_{X_2}(-Z)\otimes L)$.

    To show $h^3(X_2,\O_{X_2}(-Z)\otimes L) = 0$, first note that
    \[
        \O_{X_2}(-Z)\otimes L = \O_{X_2}((n-c)\xi_2 + ((2+\epsilon)n-b)\xi_1 + aH).
    \]
    Since $\xi_2+(2+\epsilon)\xi_1$ is ample relative to $\pi:X_2\to\PP^2$ (\Cref{lemma:Demailly3.3}(b)), it follows that for sufficiently large $n$, $\O_{X_2}(-Z)\otimes L$ is relatively ample as well, and hence $R^q\pi_*(\O_{X_2}(-Z)\otimes L) = 0$ for all $q \geq 1$.
    Leray's spectral sequence then writes
    \[
        H^3(X_2,\O_{X_2}(-Z)\otimes L)\cong H^3(\PP^2, \pi_*(\O_{X_2}(-Z)\otimes L)).
    \]
    The latter group is zero because $3 > \dim \PP^2$.

    To show $h^2(X_2,L) = 0$, we use a version of Bogomolov vanishing, which we spell out now.
    First, we push forward to $X_1$.
    Since $R^1\pi_{2*}L = 0$, 
    \[
        H^2(X_1, \pi_{2*} L) = H^2(X_1, \sym^n V_1^{\vee}\otimes\O_{X_1}((2+\epsilon)n\xi_1)).
    \]
    The dual of the sequence (\ref{ses:def V_1}) gives a filtration of $\pi_{2*}L$ whose graded pieces are the line bundles $Q_i = ((3+\epsilon)n-3i)\xi_1 +i(d-3)H$ for $0 \leq i \leq n$.
    It suffices to show that $H^2(X_1, Q_i) = 0$ for all $i$ in the interval from 0 to $n$.
    Pushing forward to $\PP^2$, the bundles $R^1\pi_{1*}Q_i$ will again vanish, so we see that 
    \[
        H^2(\PP^2, \pi_{1*}Q_i) = H^2(\PP^2, \sym^{(3+\epsilon)n-3i} \Omega_{\PP^2}(\log C)\otimes \O_{\PP^2}(i(d-3))).
    \]
    Using Serre duality, it is equivalent to show that $\sym^{(3+\epsilon)n-3i} T_{\PP^2}(-\log C)\otimes\O_{\PP^2}(-3-i(d-3))$ has no global sections for $i$ between 0 and $n$.
    Since $T_{\PP^2}(-\log C)$ is semistable of slope $-(d-3)/2$ \cite[Lemma 3]{UedaYoshinaga2008}, it follows that there are no such global sections.
    This finishes the proof.
\end{proof}

We would like to point out that for the inequality \eqref{ineq:negative_section} to possibly hold true, $4d^2-51d+90$ must be a positive number, which happens for $d\geq 11$ and hence the assumption on the degree of $C$.

\begin{prop}[cf. {\cite[Lemma 1.4.2]{el2003logarithmic}}]\label{prop:m>=6}
    Let $(\PP^2,C)$ be a logarithmic pair with $C$ smooth of degree $d \geq 15$.
    Assume $\sigma\in H^0(X_2,\O_{X_2}(m\xi_2-aH))$ is a section with $\{\sigma = 0\} = Z\cup\Gamma_2$, where $Z$ is irreducible with $[Z] = c\xi_2 + b\xi_1 - aH$.
    If $b + c \geq 6$, then 
    \[
        \frac{a}{b + c} < \frac{4d^2-51d+90}{12(d-3)}.
    \]
\end{prop}
\begin{proof}
    Note that $m = b + c$. 
    Write $m = 3p + q$, with $p$, $q$ nonnegative integers and $q\in\{0,1,2\}$. 
    Since $p\geq 2$, it follows as in \cite[Lemma 4.1]{demailly2000hyperbolicity} that there is a map
    \[
        \Delta: E_{2,m}\Omega_{\PP^2}(\log C)\otimes\O_{\PP^2}(-a) \to \sym^{(p-1)(3p+2q)}\Omega_{\PP^2}(\log C)\otimes\O_{\PP^2}(p(p-1)(d-3)-2a(p-1)).
    \]
    For a germ $s$ of a section of $E_{2,m}\Omega_{\PP^2}(\log C)$, the zero locus $\{s = 0\}$ defines a germ of a divisor $D \subset X_2$, and the zero locus of the image $\{\Delta(s) = 0\}$ is the divisor in $X_1$ along which the projection $D \to X_1$ has branch points \cite[p.533]{demailly2000hyperbolicity}.
    Since $c > 0$, the projection $Z\to X_1$ is a dominant morphism branched along a proper closed subset of $X_1$.
    Hence, we obtain a nontrivial global section $\Delta(\sigma)\in H^0(X_1, \O_{X_1}(t\xi_1 + rH))$, where
    \[
        t = (p-1)(3p+2q) \qquad\text{and}\qquad r = p(p-1)(d-3) - 2a(p-1).
    \] 
    By \Cref{thm:eff_cone}, $t \leq r$ which gives
    \[
        \frac{a}{b+c} = \frac{a}{m} \leq \frac{p(d-3) - 3p - 2q}{2m} = \frac{pd}{2m} - 1 \leq \frac{d}{6}-1.
    \]
    One can check that for $d\geq 15$, 
    \[
        \frac{d}{6}-1 < \frac{4d^2-51d+90}{12(d-3)}
    \]
\end{proof}

A consequence of \Cref{prop:m>=6} is a surprisingly simple test to verify the {\GGL}  Conjecture for complements of smooth curves of degree at least $15$:

\begin{prop}\label{elgoul test}
    Let $(\PP^2,C)$ be a logarithmic pair with $C$ smooth of degree $d \geq 15$.
    If 
    \[
        H^0(X_2, \O_{X_2}(\xi_2 + b\xi_1 - aH)) \cong H^0(X_1, V_1^\vee\otimes\O_{X_1}(b\xi_1 - aH)) = 0,
    \]
    for all $0 < b \leq 4$ and $a > 0$,
    then $\PP^2\setminus C$ satisfies {\GGL}.
\end{prop}
\begin{proof}
    As $d\geq 15$, there are positive integers $m, a$ and a section $\sigma\in H^0(X_2,\O_{X_2}(m\xi_2-aH))$ with $\{\sigma = 0\} = Z\cup\Gamma_2$, where $Z$ is irreducible (\Cref{prop:ZcupGamma}).
    Writing $[Z] = c\xi_2 + b\xi_1 - aH$, we must have $b\geq 2c\geq 0$ and $b + c = m$.
    Note that $c\neq 0$ because $b\xi_1 - aH$ is not effective.
    If $c = 1$, then $b\geq 5$ because otherwise, $Z$ would be associated with a nonzero section of $H^0(X_2, \O_{X_2}(\xi_2 + b\xi_1 - aH)\big)$ with $b\leq 4$.
    If $c\geq 2$, then $b\geq 4$.
    In any case, $m\geq 6$ and we are done by \Cref{prop:m>=6} together with \Cref{prop:restriction_is_big}.
\end{proof}

Combining the above results we obtain \Cref{test theorem}, which provides a test we can use to verify {\GGL} for complements of smooth plane curves relying only on the effectivity of the family of line bundles $\O_{X_2}(\xi_2 + b\xi_1 - aH)$, where the coefficient of the class $\xi_2$ is equal to one.
\begin{proof}[Proof of \Cref{test theorem}]
    If $d\geq 15$ and $H^0(X_2, \O(\xi_2 + b\xi_1 - aH)) = 0$ for all $a, b > 0$, {\GGL} follows from \Cref{elgoul test}.
    If $d\geq 11$, $H^0(X_2, \O(\xi_2 + b\xi_1 - aH))$ has one section whose vanishing locus is irreducible, and 
    \[
        \frac{a}{b + 1} < \frac{4d^2-51d+90}{12(d-3)},
    \]
    then \Cref{prop:restriction_is_big} gives {\GGL}.
    Finally, if there exist two sections in the spaces $H^0(X_2, \O(\xi_2 + b\xi_1 - aH))$ for possibly different values of $a$ and $b$, not both sections vanishing along the same divisor, their common zero locus in $X_2$ has dimension at most 2, and {\GGL} follows from \Cref{thm:McQuillan}.
\end{proof}

\section{Computing 2-Jet Differentials via Bigraded Modules}\label{sec:Explicit Computations}
In this section we give an algorithm for computing sections of
\begin{equation}\label{iso: interesting sections}
    H^0(X_2, \O(\xi_2 + b\xi_1 - aH)) \cong H^0(X_1, V_1^\vee\otimes\O(b\xi_1 - aH)) \qquad\text{ for all } a,b > 0,
\end{equation}
which together with \Cref{test theorem}, becomes a computational test for {\GGL} on complements of smooth plane curves.
The main step in this direction is to express both the vector bundle $T_{X_1}(-\log\pi_1^*C)$ and the subbundle $V_1$ in terms of maps between line bundles (\Cref{TX1log as maps of line bundles}).
The representation of $X_1$ as a hypersurface in $\PP^2\times\PP^2$ is crucial for this step.

In order to write $T_{X_1}(-\log\pi_1^*C)$ and $V_1$ in terms of maps of line bundles, we set some notation.
Let $R = \C[x_0,x_1,x_2,a_0,a_1,a_2]$ be the bigraded polynomial ring with $\deg(x_i) = (1,0)$ and $\deg(a_i) = (0,1)$.
Let $S = R/(g)$ be the bihomogeneous coordinate ring of $X_1$, so that $X_1 = \Proj(S) \subset \PP^2\times\PP^2$.
Furthermore, we will use the bigraded notation based on \eqref{eq:pullback classes}:
\[
    \O_{X_1}(1,0) = \O_{X_1}(H) \qquad\text{and}\qquad \O_{X_1}(0,1) = \O_{X_1}(\xi_1 + H).
\]
Therefore, twists of any vector bundle $\E$ on $X_1$ are written $\E(p,q) = \E\otimes\O_{X_1}(q\xi_1+(p+q)H)$.
Moreover, translating \Cref{thm:eff_cone} into bigraded notation, if $H^0(X_1,\O_{X_1}(p,q)) \neq 0$, then $p+q \geq q \geq 0$, i.e., both $p$ and $q$ must be nonnegative.

Pulling back the sequence \eqref{seq:euler_log} to $X_1$ we get the exact sequence 
\[\begin{tikzcd}
    0 \rar & \pi_1^*T_{\PP^2}(-\log C) \rar & \O_{X_1}(1,0)^3 \rar & \O_{X_1}(d,0)\rar & 0
\end{tikzcd}\]
and thus we consider $\pi_1^*T_{\PP^2}(-\log C)$ as a subsheaf of $\O_{X_1}(1,0)^3$.
Consequently, the tautological line $\O_{X_1}(-\xi_1) = \O_{X_1}(1,-1)\subset \pi_1^*T_{\PP^2}(-\log C)$ can be embedded in $\O_{X_1}(1,0)^3$.
At each $(x,a)\in X_1$, the vector $a = (a_0,a_1,a_2)$ is the generator of the tautological line in the fiber of $\O_{X_1}(1,0)^3$ over $x$, thus the embedding is given by
\[
    \O_{X_1}(1,-1)\xhookrightarrow{[a_0, a_1, a_2]^T} \O_{X_1}(1,0)^3.
\]

With this bigraded notation, we can use duality to obtain global sections of \eqref{iso: interesting sections} as global sections of bigraded twists of $V_1$.
\Cref{lemma:sections of V1} states the precise relationship.
\begin{lemma}\label{lemma:sections of V1}
    Let $(\PP^2,C)$ be a smooth logarithmic pair of general type with $C$ of degree $d$.
    Duality gives an isomorphism
    \[
        H^0(X_1, V_1(p,q)) \cong H^0(X_1, V_1^\vee(p-d+2,q+1)).
    \]
    In particular, global sections of $V_1(p,q)$ correspond to negatively twisted logarithmic invariant 2-jet differentials if and only if $p+q \leq d-4$.
\end{lemma}
\begin{proof}
    Note that $\det(V_1)\cong\O_{X_1}(2-d,1)$.
    As $V_1$ is locally free of rank 2, we have $V_1\cong V_1^\vee\otimes\det(V_1)$ and therefore, 
    \[
         V_1(p,q) \cong V_1^\vee(p-d+2,q+1).
    \]
    The isomorphism in global sections follows.
    
    The bigraded notation writes
    \[
        V_1^\vee(p-d+2,q+1) = V_1^\vee((q+1)\xi_1 + (p+q-d+3)H).
    \]
    Therefore, global sections of $V_1(p,q)$ are dual to negatively twisted logarithmic invariant 2-jet differentials if and only if $p+q \leq d-4$ and $q\geq -1$.
    However, we claim that $H^0(X_1,V_1(p,q)) = 0$ if $p+q \leq d-4$ and $q < -1$, making the ``in particular'' statement trivially true.
    To show the claim, consider the exact sequence \eqref{ses:def V_1} twisted by $\O_{X_1}(p,q)$:
    \begin{equation}\label{ses: def V1 bigraded}
    \begin{tikzcd}
        0 \rar & \O_{X_1}(p-d+1,q+2)\rar & V_1(p,q)\rar & \O_{X_1}(p+1,q-1)\rar & 0.
    \end{tikzcd}
    \end{equation}
    By \Cref{thm:eff_cone}, the kernel has no global sections if $p+q \leq d-4$, and the quotient has no global sections if $q < -1$.
\end{proof}

Once we obtain the description of $T_{X_1}(-\log\pi_1^*C)$ and $V_1$ in terms of maps between line bundles (\Cref{TX1log as maps of line bundles}), we use this description in \Cref{V1 as maps of modules} to compute a graded section module that captures all the sections of $V_1$ giving rise to negatively twisted logarithmic invariant 2-jet differentials.
Its proof relies on \Cref{lemma: injectivity of H1 map} below to argue that the module description captures all the global sections in each twist by $\O_{X_1}(p,q)$.

\begin{lemma}\label{lemma: injectivity of H1 map}
    Assume $(p,q)\in\Z^2$ with $q\geq -1$, and let $\O_{X_1}(p,q) \to \O_{X_1}(p,q+1)$ be the map given by multiplication by $a_i$, for some $i\in\{0,1,2\}$.
    Then the induced map on first cohomology groups
    \[
        \alpha: H^1(X_1, \O_{X_1}(p,q)) \to H^1(X_1,\O_{X_1}(p,q+1))
    \]
    is injective.
\end{lemma}
\begin{proof}
    To prove this, we want to compare $\alpha$ to a map on $Y = \PP^2 \times \PP^2$. Consider the following diagram
    \begin{equation}\label{diagram:restriction}
    \begin{tikzcd}
        0\dar & 0\dar \\
        \O_Y(p-d+1,q-1) \dar\rar & \O_Y(p-d+1,q)\dar \\
        \O_Y(p,q) \dar\rar & \O_Y(p,q+1)\dar \\
        \O_{X_1}(p,q) \rar\dar & \O_{X_1}(p,q+1)\dar\\
        0 & 0 
    \end{tikzcd}
    \end{equation}
    where all horizontal maps are given by multiplication by $a_i$. Note that, by the K\"unneth formula, $H^1(\PP^2\times\PP^2, \O(p,q)) = 0$. Thus, taking higher cohomology of (\ref{diagram:restriction}) gives the following.
    \begin{equation*}
    \begin{tikzcd}
        H^1(X_1, \O_{X_1}(p,q)) \rar{\alpha}\dar[hook] & H^1(X_1,\O_{X_1}(p,q+1))\dar[hook]\\
        H^2(Y, \O_Y(p-d+1,q-1)) \rar & H^2(Y, \O_Y(p-d+1,q))
    \end{tikzcd}
    \end{equation*}
    The injectivity of the top map follows from the injectivity of the bottom map.

    To see injectivity of the bottom map, note that the map is induced from the map on cohomology of the sequence
    \[\begin{tikzcd}
        0 \rar & \O_Y(p-d+1,q-1) \rar & \O_Y(p-d+1,q) \rar & \O_{\PP^2 \times \PP^1}(p-d+1,q) \rar & 0,
    \end{tikzcd}\]
    and so the kernel is $H^1(\PP^2\times\PP^1, \O_{\PP^2 \times \PP^1}(p-d+1,q))$. 
    This last is 0 by the K\"unneth formula and the fact that $q \geq -1$.
    The result follows.
\end{proof}

\subsection{A Bigraded Section Module for 2-Jet Differentials}
Because for the rest of the paper we will deal exclusively with sheaves on $X_1$, we will omit the subscript from the structure sheaf $\O_{X_1}$.

\begin{prop}\label{TX1log as maps of line bundles}
    Let $X_1 = \Proj(S) \subset \PP^2 \times \PP^2$ with projection $\pi_1:X_1\to \PP^2$ to the first factor.
    \begin{enumerate}
        \item[(a)] Let $\K$ be the kernel of the map $\O(1,0)^3\oplus\O(0,1)^3 \to \O(d,0)\oplus\O(d-1,1)$ defined by the matrix
        \[\begin{bmatrix}
            \frac{\partial f}{\partial x_0} & \frac{\partial f}{\partial x_1} & \frac{\partial f}{\partial x_2} & 0 & 0 & 0\\[5pt]
            \frac{\partial g}{\partial x_0} & \frac{\partial g}{\partial x_1} & \frac{\partial g}{\partial x_2} & \frac{\partial f}{\partial x_0} & \frac{\partial f}{\partial x_1} & \frac{\partial f}{\partial x_2} 
        \end{bmatrix}.
        \]
        Then,
        \[
            T_{X_1}(-\log\pi_1^*C) \cong \coker(\O\xrightarrow{[0 \ 0 \ 0 \ a_0 \ a_1 \ a_2]^T}\K).
        \]
        \item[(b)] Let $\K'$ be the kernel of the map $\O(1,-1)\oplus\O(0,1)^3 \to \O(d-1,1)$ defined by the matrix
        \[\begin{bmatrix}
            \sum a_i\frac{\partial g}{\partial x_i} & \frac{\partial f}{\partial x_0} & \frac{\partial f}{\partial x_1} & \frac{\partial f}{\partial x_2}
        \end{bmatrix}.\]
        Then
        \[
            V_1 \cong \coker(\O\xrightarrow{[0 \ a_0 \ a_1 \ a_2]^T}\K').
        \]
    \end{enumerate}
\end{prop}
\begin{proof}
    Let $\widetilde{C}$ be the preimage of $C$ inside $\PP^2\times\PP^2$ via the first projection.
    We claim that
    \[
        T_{X_1}(-\log\pi_1^*C) \cong T_{\PP^2\times\PP^2}(-\log\widetilde{C})\big|_{X_1} \cap T_{X_1},
    \]
    where the intersection on the right-hand side occurs inside $T_{\PP^2\times\PP^2}$. 
    As $X_1\subset\PP^2\times\PP^2$ is a smooth hypersurface defined by the polynomial
    \[
        g = a_0 \frac{\partial f}{\partial x_0} + a_1 \frac{\partial f}{\partial x_1} + a_2 \frac{\partial f}{\partial x_2},
    \]
    we have the following short exact sequence:
    \begin{equation}\label{ses:normal T_{X_1}}
    \begin{tikzcd}
        0 \rar & T_{X_1} \rar & T_{\PP^2\times\PP^2}\big|_{X_1}\rar & \O(d-1,1)\rar & 0.
    \end{tikzcd}
    \end{equation}
    Combining it with the exact sequences defining logarithmic tangent sheaves, we get a commutative diagram.
    \[\begin{tikzcd}
        & 0\dar & 0\dar & &\\
        & T_{X_1}(-\log\pi_1^*C)\dar & T_{\PP^2\times\PP^2}(-\log\widetilde{C})\big|_{X_1}\dar & \\
        0\rar & T_{X_1}\dar\rar & T_{\PP^2\times\PP^2}\big|_{X_1}\dar\rar & \O(d-1,1)\rar & 0\\
        & \O_{\pi_1^*C}(\pi_1^*C) \dar\rar{\cong} & \O_{\widetilde{C}}(\widetilde{C})\big|_{X_1} \dar & &\\
        & 0 & 0 &&
    \end{tikzcd}\]
    The Snake Lemma completes the diagram
    \[\begin{tikzcd}
        & 0\dar & 0\dar & &\\
        0\rar & T_{X_1}(-\log\pi_1^*C)\dar\rar & T_{\PP^2\times\PP^2}(-\log\widetilde{C})\big|_{X_1}\dar\rar & \O(d-1,1)\dar{\cong}\rar & 0\\
        0\rar & T_{X_1}\dar\rar & T_{\PP^2\times\PP^2}\big|_{X_1}\dar\rar & \O(d-1,1)\rar & 0\\
        & \O_{\pi_1^*C}(\pi_1^*C) \dar\rar{\cong} & \O_{\widetilde{C}}(\widetilde{C})\big|_{X_1} \dar & &\\
        & 0 & 0 &&
    \end{tikzcd}\]
    showing that indeed $T_{X_1}(-\log\pi_1^*C) \cong T_{\PP^2\times\PP^2}(-\log\widetilde{C})\big|_{X_1} \cap T_{X_1}$.
    
    As with the diagram in \eqref{diagram1}, we have three short exact sequences which we can write in a single diagram as below. 
    \[\begin{tikzcd}
        && 0\dar & 0\dar & \\
        && \O\oplus\O\dar & \O\dar & \\
        && \O(1,0)^3 \oplus \O(0,1)^3\dar & \O(d,0)\dar & \\
        0\rar & T_{\PP^2\times\PP^2}(-\log \widetilde{C})\big|_{X_1}\rar & T_{\PP^2\times\PP^2}\big|_{X_1}\dar\rar & \O_{\widetilde{C}}(\widetilde{C})\big|_{X_1}\dar\rar & 0\\
        && 0 & 0 &
    \end{tikzcd}
    \]
    Consider the map
    \[
       \nabla f: \O(1,0)^3 \oplus\O(0,1)^3 \to \O(d,0), \qquad \nabla f = \begin{bmatrix} \frac{\partial f}{\partial x_0} & \frac{\partial f}{\partial x_1} & \frac{\partial f}{\partial x_2} & 0 & 0 & 0\end{bmatrix}.
    \]
    Also, let $\O\oplus\O\to \O\oplus 0 = \O$ be the projection onto the first factor.
    These two maps commute with the maps in the diagram above and thus, by the Snake Lemma, we get a commutative diagram with exact columns and rows.
    \[\begin{tikzcd}
        & 0\dar & 0\dar & 0\dar & \\
        0\rar & 0\oplus\O\dar\rar & \O\oplus\O\dar\rar & \O\oplus 0 \dar\rar & 0 \\
        0\rar & \ker(\nabla f)\rar\dar & \O(1,0)^3 \oplus \O(0,1)^3\dar\rar{\nabla f} & \O(d,0)\dar\rar & 0 \\
        0\rar & T_{\PP^2\times\PP^2}(-\log \widetilde{C})\big|_{X_1}\dar\rar & T_{\PP^2\times\PP^2}\big|_{X_1}\dar\rar & \O_{\widetilde{C}}(\widetilde{C})\big|_{X_1}\dar\rar & 0\\
        & 0 & 0 & 0 &
    \end{tikzcd}
    \]
    As a result, $T_{\PP^2\times\PP^2}(-\log \widetilde{C})\big|_{X_1} \cong \ker(\nabla f)/(0\oplus\O)$.
    The composition 
    \[
        0\oplus\O \longrightarrow \O\oplus\O \longrightarrow \O(1,0)^3 \oplus \O(0,1)^3,
    \]
    and thus the inclusion $0\oplus\O\hookrightarrow\ker(\nabla f)$, is defined by the matrix $\begin{bmatrix}0 & 0 & 0 & a_0 & a_1 & a_2\end{bmatrix}^T$.

    The Jacobian of $g$ defines a map $\nabla g: \O(1,0)^3 \oplus \O(0,1)^3\to \O(d-1,1)$ that factors through $T_{\PP^2\times\PP^2}\big|_{X_1}$.
    Together with \eqref{ses:normal T_{X_1}}, this forms a commutative diagram.
    \[\begin{tikzcd}
         & & 0 \dar & & \\
         & & \O\oplus\O \dar & & \\
         & & \O(1,0)^3\oplus\O(0,1)^3 \dar\rar{\nabla g} & \O(d-1,1)\rar\dar[equal] & 0\\
         0 \rar & T_{X_1}\rar & T_{\PP^2 \times\PP^2}\big|_{X_1} \dar\rar & \O(d-1,1) \rar & 0\\
         & & 0 & &
    \end{tikzcd}\]
    From this diagram, it follows that $T_{X_1}\cong \ker(\nabla g)/(\O\oplus\O)$.
    On the other hand, $\ker(\nabla f) \cap (\O\oplus 0) = 0$ and so
    \begin{align*}
        (\ker(\nabla f) \cap \ker(\nabla g))/(0\oplus\O) \cong (\ker(\nabla f)/(0\oplus\O)) \cap (\ker(\nabla g)/(\O\oplus\O)) \\
        \subseteq (\O(1,0)^3 \oplus \O(0,1)^3)/(\O\oplus\O).
    \end{align*}
    That is,
    \begin{equation}\label{iso: intersection ker nabla}
        (\ker(\nabla f) \cap \ker(\nabla g))/(0\oplus\O) \cong T_{\PP^2\times\PP^2}(-\log\widetilde{C})\big|_{X_1} \cap T_{X_1} \cong T_{X_1}(-\log\pi_1^*C).
    \end{equation}

    The map
    \[
        \nabla f\oplus \nabla g: \O(1,0)^3\oplus\O(0,1)^3 \longrightarrow \O(d,0)\oplus\O(d-1,1)
    \]
    is defined by the matrix
    \[\begin{bmatrix}
        \frac{\partial f}{\partial x_0} & \frac{\partial f}{\partial x_1} & \frac{\partial f}{\partial x_2} & 0 & 0 & 0\\[5pt]
        \frac{\partial g}{\partial x_0} & \frac{\partial g}{\partial x_1} & \frac{\partial g}{\partial x_2} & \frac{\partial g}{\partial a_0} & \frac{\partial g}{\partial a_1} & \frac{\partial g}{\partial a_2}
    \end{bmatrix}
    =
    \begin{bmatrix}
        \frac{\partial f}{\partial x_0} & \frac{\partial f}{\partial x_1} & \frac{\partial f}{\partial x_2} & 0 & 0 & 0\\[5pt]
        \frac{\partial g}{\partial x_0} & \frac{\partial g}{\partial x_1} & \frac{\partial g}{\partial x_2} & \frac{\partial f}{\partial x_0} & \frac{\partial f}{\partial x_1} & \frac{\partial f}{\partial x_2}
    \end{bmatrix}
    \]
    and we have the isomorphism $\ker(\nabla f\oplus \nabla g) \cong \ker(\nabla f) \cap \ker(\nabla g)$.
    Denoting $\K = \ker(\nabla f\oplus \nabla g)$, the isomorphism in \eqref{iso: intersection ker nabla} yields the short exact sequence
    \begin{equation}\label{ses:TX1(-log)}
    \begin{tikzcd}
        0 \rar & \O \arrow{rrr}{[0 \ 0 \ 0 \ a_0 \ a_1 \ a_2]^T} & & & \K \rar & T_{X_1}(-\log\pi_1^*C) \rar & 0.
    \end{tikzcd}
    \end{equation}

    We proceed to prove part (b).
    Via projection, we have the following map of short exact sequences.
    \begin{equation}\label{ses: projection for K}
    \begin{tikzcd}
        0 \rar & \mathcal{K} \ar[d]\rar & \O(1,0)^3 \oplus \O(0,1)^3 \ar[d]\rar & \O(d,0) \oplus \O(d-1,1) \ar[d] \rar & 0 \\
        0 \rar & \pi_1^*T_{\PP^2}(-\log C) \rar & \O(1,0)^3 \rar & \O(d,0) \rar & 0 \\
    \end{tikzcd}
    \end{equation}
    Recall that $\O(1,-1)\subset \pi_1^*T_{\PP^2}(-\log C)$. 
    The preimage $\mathcal{K}'$ of $\O(1,-1)$ in $\mathcal{K}$ is, at least a priori, the kernel of the composition 
    \[
        \O(1,-1) \oplus \O(0,1)^3 \longrightarrow \O(1,0)^3 \oplus \O(0,1)^3 \longrightarrow \O(d,0) \oplus \O(d-1,1).
    \]
    However, from the second row in \eqref{ses: projection for K}, the part landing in the $\O(d,0)$ summand is always zero.
    It follows that $\K'$ is defined by the sequence
    \[\begin{tikzcd}
        0 \rar & \mathcal{K}' \rar & \O(1,-1) \oplus \O(0,1)^3 \rar & \O(d-1,1) \rar & 0.
    \end{tikzcd}\]
    The map $\O(1,-1) \to \O(d-1,1)$ is simply the composition
    \[
        \O(1,-1)\xrightarrow{[a_0 \ a_1 \ a_2]^T}\O(1,0)^3 \xrightarrow{[\frac{\partial g}{\partial x_0} \ \frac{\partial g}{\partial x_1} \ \frac{\partial g}{\partial x_2}]} \O(d-1,1).
    \]
    In other words, it is the map defined by $\sum a_i \frac{\partial g}{\partial x_i}$. 
    Since $V_1 \subset T_{X_1}(-\log \pi_1^*C) = \K / \O$ is the preimage of the universal sub-bundle and the natural map $\O \to \mathcal{K}$ lands in $\mathcal{K}'$, it follows that $V_1$ will be $\mathcal{K}' / \O$, as claimed.
\end{proof}

\begin{cor}\label{V1 as maps of modules}
    Let $X_1 = \Proj(S) \subset \PP^2\times\PP^2$ and let $K'$ be the kernel of the map $S(1,-1)\oplus S(0,1)^3 \to S(d-1,1)$ defined by the matrix
    \[
        \begin{bmatrix}
            \sum a_i\frac{\partial g}{\partial x_i} & \frac{\partial f}{\partial x_0} & \frac{\partial f}{\partial x_1} & \frac{\partial f}{\partial x_2}
        \end{bmatrix}.
    \]
    There is an $S$-module isomorphism
    \[
        \bigoplus_{\substack{(p,q)\in\Z^2\\q\geq -1}} H^0\big(X_1,V_1\otimes\O(p,q)\big) \cong \coker(S\xrightarrow{[0 \ a_0 \ a_1 \ a_2]^T} K')
    \]
\end{cor}
\begin{proof}
    Note that for all $(p,q)\in\Z^2$, $H^0(X_1,\O(p,q)) = S_{(p,q)}$.
    Hence, the short exact sequence
    \[\begin{tikzcd}
        0\rar & \K' \rar & \O(1,-1)\oplus\O(0,1)^3 \rar & \O(d-1,1) \rar & 0
    \end{tikzcd}\]
    and the left exactness of the global sections functor give the $S$-module isomorphism
    \[
        H^0(X_1,\K'(p,q)) \cong K'_{(p,q)}.
    \]

    Consider the short exact sequence
    \[\begin{tikzcd}
        0 \rar & \O \rar & \K' \rar & V_1 \rar & 0.
    \end{tikzcd}\]
    Twisting and taking higher cohomology we obtain the following exact sequence.
    \[\begin{tikzcd}
        0\rar & S_{(p,q)} \rar & K'_{(p,q)} \rar & H^0(X_1,V_1(p,q)) \dar{\delta} \\
        & & & H^1(X_1,\O(p,q)) \rar{\alpha} & H^1(\K'(p,q)) \rar & \cdots
    \end{tikzcd}\]

    Since 
    \[
        \big(\coker(S_{(p,q)}\to K'_{(p,q)})\big)_{(p,q)} \cong \big(\coker(S\to K')\big)_{(p,q)},
    \]
    we are left to prove $\im(\delta) = 0$, or equivalently $\ker(\alpha) = 0$, for all $(p,q)\in\Z^2$ with $q\geq -1$.

    From \Cref{TX1log as maps of line bundles}, the composition $\O\to\K'\to\O(1,-1)\oplus\O(0,1)^3$ is defined by the matrix $[0 \ a_0 \ a_1 \ a_2]^T$.
    This map induces the commutative diagram below.
    \[\begin{tikzcd}
        H^1(X_1,\O(p,q)) \rar\dar{\alpha} & H^1(X_1,\O(p+1,q-1))\oplus H^1(X_1,\O(p,q+1))^3\\
        H^1(X_1,\K'(p,q)) \ar[ur]
    \end{tikzcd}\]
    Since $q\geq -1$, the top map is injective by \Cref{lemma: injectivity of H1 map}.
    It follows that 
    \[\begin{tikzcd}
        0\rar & S_{(p,q)} \rar & K'_{(p,q)} \rar & H^0(X_1,V_1(p,q))\rar & 0
    \end{tikzcd}\]
    is exact.
    Since $H^0(X_1,\O(p,q)) = S_{(p,q)}$ has nonzero sections for $p$ and $q$ nonnegative, the degree-wise isomorphisms are compatible with the $S$-action.
    The result follows.
\end{proof}

It is worth noting that the only sections of $V_1(p,q)$ that are not captured by the corresponding graded piece of the $S$-module $\coker(S\to K')$ are those in degree $(p,-2)$ with $p \geq d-1$.
These correspond to sections of the line bundle $\O_{X_2}(\xi_2 - \xi_1 + (p-d+1)H)$, which do not provide any restriction on lifts of entire curves.


\subsection{The Computational Test and Its Implementation}

Here we provide a Macaulay2 script for computing the $S$-module $\coker(S\to K')$.
By \Cref{lemma:sections of V1} together with \Cref{V1 as maps of modules}, this module contains the entire family of global sections in \eqref{iso: interesting sections}.
As with all modules, Macaulay2 represents it in terms of generators and relations.
Hence, we will translate \Cref{test theorem} using \Cref{lemma:sections of V1} to test for {\GGL} based on a minimal set of homogeneous generators for $\coker(S\to K')$.

\begin{defn}
    Let $S$ be the homogeneous coordinate ring of $X_1$.
    \begin{enumerate}
        \item With the same notation as in \Cref{V1 as maps of modules}, denote by 
        \[
            M_V = \coker(S\xrightarrow{[0 \ a_0 \ a_1 \ a_2]^T} K').
        \]
        \item Given an element $\sigma\in (M_V)_{(p,q)}$ with $p+q\leq d-4$, denote by $Z_\sigma$ the closed subset of $X_2$ that is the vanishing locus of the section in $H^0(X_2,\O(\xi_2+(q+1)\xi_1 - (d-3-p-q)H))$ corresponding to $\sigma$ (cf. \Cref{lemma:sections of V1}).
    \end{enumerate}
\end{defn}
\begin{lemma}\label{lemma: Zalpha irreducible}
    If $p+q\leq d-4$ and $\sigma\in(M_V)_{(p,q)}$ is part of a set of minimal homogeneous generators for $M_V$, then $Z_\sigma \subset X_2$ is irreducible. 
\end{lemma}
\begin{proof}
    If $Z_\sigma$ is the union of two effective divisors $Z_1 \cup Z_2$, then
    \[
        Z_1 \in \big\lvert\xi_2 + (q+1-r)\xi_1 - (d-3-p-q+t)H\big\rvert \qquad\text{and}\qquad Z_2 \in \big\lvert r\xi_1 + tH\big\rvert
    \]
    where $(r,t)\neq (0,0)$.
    Hence, there must be another element $\sigma_0\in M_V$ and a degree $(t-r,r)$ element $s\in S$ such that $\sigma = s\sigma_0$, contradicting the fact that $\sigma$ is a minimal generator.
\end{proof}

Translating by duality the conditions of \Cref{test theorem}, we obtain a computational test to confirm the {\GGL} Conjecture for the complement of a smooth plane curve based on a minimal set of homogeneous generators for $M_V$.

\begin{thm}[Equivalent to \Cref{test theorem}]\label{Module test}
    Let $f\in\C[x_0,x_1,x_2]$ be a homogeneous polynomial of degree at least $4$ defining a smooth plane curve $C \subset \PP^2$.
    Let $n$ be the number of homogeneous generators of total degree $p+q\leq d-4$ in a minimal homogeneous generating set for $M_V$.
    Then $\PP^2\setminus C$ satisfies {\GGL} if one of the following conditions is satisfied:
    \begin{enumerate}
        \item[(I)] $n=0$ and $d\geq 15$.
        \item[(II)] $n=1$, $d\geq 11$, and the only minimal homogeneous generator of total degree $p+q\leq d-4$ satisfies the inequality
        \begin{equation}\label{ineq:module_too_negative_section}
            \frac{d-3-(p+q)}{q+2} < \frac{4d^2-51d+90}{12(d-3)}.
        \end{equation}
        \item[(III)] $n\geq 2$ and $d\geq 4$. 
    \end{enumerate}
\end{thm}
\begin{proof}
    If (I) holds, then
    \[
        H^0\big(X_2,\O(\xi_2 + (q+1)\xi_1 - (d-3-p-q)H)\big) = 0 
    \]
    for $q+1 > 0$ and $d-3-p-q > 0$, which is equivalent to criterion (I) in \Cref{test theorem}.
    If (II) holds, then we get a global section of $\O(\xi_2 + (q+1)\xi_1 - (d-3-p-q)H)$ satisfying the inequality \eqref{ineq:negative_section}, which is exactly criterion (II) in \Cref{test theorem}.
    Finally, if (III) holds, then we have two negatively twisted logarithmic invariant 2-jet differentials.
    By \Cref{lemma: Zalpha irreducible}, they do not vanish along a common divisor in $X_2$, and this is criterion (III) in \Cref{test theorem}.
\end{proof}

For degree $d\geq 15$ our test in \Cref{Module test} is inconclusive only in the case where $M_V$ has only one minimal homogeneous generator of total degree at most $d-4$ and the degree of this generator does not satisfy the inequality \eqref{ineq:module_too_negative_section}.
Moreover, we know from \Cref{prop:m>=6} that such a generator must be in degree $(p,q)$ with $q \leq 3$.
As such, one can reduce {\GGL} for the complements of smooth plane curves of degree at least 15 to the study of those belonging to the aforementioned case.
For $d < 15$, one must also consider the case in which $M_V$ has no minimal generators of degree at most $d-4$, i.e., the case $n = 0$, with $n$ is as in \Cref{Module test}.

Using \Cref{V1 as maps of modules}, we can implement a Macaulay2 script that explicitly computes a minimal homogeneous generating set for $M_V$ given a homogeneous polynomial $f$ defining the curve $C$.
Below is the Macaulay2 code with $f = x_0^{16} + x_1^{16} + x_2^{16}$, the Fermat curve of degree $d = 16$, which prints the minimal homogeneous generators of $M_V$ together with their bidegrees.
To perform the computation for another plane curve, one only needs to modify the definitions of \texttt{d} and \texttt{f} in Listing 1 below.

\begin{lstlisting}[caption={Macaulay2 code for computing the module $M_V$.}, label={code:MV}]
R = QQ[x_0..x_2,a_0..a_2, Degrees => {3:{1,0}, 3:{0,1}}];
d = 16;
f = x_0^d + x_1^d + x_2^d;

-- Equation of X_1 inside P^2 x P^2:
g = sum apply(3, i -> a_i * diff(x_i,f));

-- First entry of the matrix in Corollary 5.4
Sigma = sum apply(3, i -> a_i * diff(x_i,g));

S = R/ideal(g);

F1 = S^{{1,-1}, 3:{0,1}};
F0 = S^{{d-1,1}};

phi = map( 
    F0, 
    F1,
    sub(matrix{{Sigma, diff(x_0,f), diff(x_1,f), diff(x_2,f)}},  S)
);

-- Definition of K' (Corollary 5.4)
Kprime = ker phi;

-- Definition of the tautological section e = (0,a_0,a_1,a_2)^T
alpha = map(
    F1, 
    S^{{0,0}},
    sub(transpose matrix{{0,a_0,a_1,a_2}}, S)
);
eS = image alpha;

-- Definition of M_V.
MV = trim (Kprime/eS);

G = mingens MV;

print "Degree of minimal generators of M_V:";
print degrees source G;
print "";

-- Retrieving index of low-degree generators
lowIndices = select(
    toList(0..numColumns(G)-1),
    i -> (
        dd := degree(G_i);
        dd#0 + dd#1 <= d-4
    )
);

print("Minimal generators of degree at most d-4 = " | toString(d-4) | ":")
scan(lowIndices, 
    i -> (
        print (degree G_i); 
        print(G_i);)
);

\end{lstlisting}

\begin{lstlisting}[style=terminal, caption={Output of the Macaulay2 computation in Listing 1}, label={code:output}]
Degree of minimal generators of M_V:
{{2, 1}, {15, -1}, {15, -1}, {15, -1}, {14, 1}, {14, 1}, {14, 1}, {14, 3}, {14, 3}, {14, 3}}

Minimal generators of degree at most d-4 = 12:
{2, 1}
|           x_0x_1x_2          |
|               0              |
| 15x_1x_2a_0a_1-15x_0x_2a_1^2 |
| 15x_1x_2a_0a_2-15x_0x_1a_2^2 |
\end{lstlisting}

In the case of the Fermat curve of degree $d = 16$, the minimal generators of total degree at most $d-4 = 12$ are the ones corresponding to invariant logarithmic 2-jet differentials and thus sections that restrict the stable base locus $B_2\subset X_2$. 
As shown in Listing 2, the only such generator is the one of degree $(2,1)$, which gives rise to a global section of 
\[
    H^0(X_2, \O_{X_2}(\xi_2 + 2\xi_1 - 10H)) \cong H^0(\PP^2, E_{2,3}\Omega_{\PP^2}(\log C)\otimes\O_{\PP^2}(-10)).
\]
We remark that this generator does not satisfy the inequality in \eqref{ineq:module_too_negative_section}.
As a consequence, the test in \Cref{Module test} is insufficient for concluding {\GGL} for the complement of the Fermat curve.
Nonetheless, it follows from \cite[Theorem 1]{toda1971functional} that the complement of the Fermat curve $C_d\subset \PP^2$ of degree $d$ satisfies {\GGL} for all $d\geq 7$.
In fact, the exceptional locus is the union of $3d$ lines:
\[
    \operatorname{Exc}(\PP^2\setminus C_d) = \bigcup_{\lambda^d = -1}\big(\{x_0 = \lambda x_1\} \cup \{x_0 = \lambda x_2\}\cup \{x_1 = \lambda x_2\}\big).
\]


\section{Examples}\label{sec:examples}
In this section we compile examples of smooth plane curves exhibiting all the possible outcomes in \Cref{Module test} regarding minimal homogeneous generators of $M_V$ in degree at most $d-4$.
The Macaulay2 files used to verify all computational examples in this paper are available in the accompanying repository \url{https://github.com/jazieltorres/jet-differentials-plane-curves}.

\subsection{Vanishing of Low-Degree Sections}
The following polynomials define smooth plane curves whose module $M_V$ has no minimal homogeneous generators of total degree at most $d-4$.
It follows from \Cref{Module test}(I) that the complements of all these curves satisfy {\GGL}. 

\begin{example}
    $f = x_0^{15} + x_1^{15} + x_2^{15} + x_0^5 x_1^5 x_2^5$.
\end{example}
\begin{example}
    $f = x_0^{15} + x_1^{15} + x_2^{15} + x_0x_1x_2(x_0 + x_1 + x_2)^{12}$.
\end{example}
\begin{example}
    $f = x_0^{17} + 2x_1^{17} + 3x_2^{17} + 13x_0^6x_1^5x_2^6$.
\end{example}
\begin{example}
    $f = x_0^{14}x_1 + x_0x_2^{14} + x_1^{14}x_2 + x_0^5 x_1^5 x_2^5$.
\end{example}

We want to emphasize that the conclusion that the complements of these curves satisfy {\GGL} is completely theoretical, based on a Riemann-Roch existence argument (cf. \Cref{prop:restriction_is_big}).
We lack any information about the multifoliation on $\PP^2\setminus C$ that allows us to conclude {\GGL}, let alone information about the exceptional locus of the complement $\PP^2\setminus C$.

\subsection{A Single Low-Degree Section}
The following polynomials define smooth plane curves whose module $M_V$ has only one generator of total degree at most $d-4$ and this generator satisfies the inequality \eqref{ineq:module_too_negative_section}.
It follows from \Cref{Module test}(II) that the complements of all these curves satisfy {\GGL}.
As in the previous examples, we do not get an explicit description of the exceptional locus.
\begin{example}
    $f = x_0^{13}+x_0^6x_1^7+x_1^{13}+x_1x_2^{12}$.
    The unique minimal generator of $M_V$ with total degree at most $d-4 = 9$ has bidegree $(8,1)$.
    This is the polynomial obtained by taking $n=6$ in the family of \Cref{example:case 3(2)} below (which assumes $n \geq 7$).
\end{example}
\begin{example}
    $f = x_0^{13}+x_0^6x_1^7+x_1^{13}+x_1x_2^{12}+x_2^{13}$.
    The unique minimal generator of $M_V$ with total degree at most $d-4 = 9$ has bidegree $(8,1)$.
\end{example}

\begin{example}
    $f = x_0^{18} + x_1^{18} + x_2^{18} + x_0^9x_1^9 + x_0^{12}x_1^4x_2^2$.
    The unique minimal generator of $M_V$ with total degree at most $d-4 = 14$ has bidegree $(11,1)$.
\end{example}

\subsection{Multiple Low-Degree Sections and Exceptional Loci}
The polynomials listed below define smooth plane curves for which the module $M_V$ has at least two minimal homogeneous generators of total degree at most $d-4$.
Consequently, for these examples the stable base locus
\[
    \widetilde{Z} = \bigcap_{(b,a)\in\N^2} \bs\big\lvert\xi_2 + b\xi_1 - aH \big\rvert \quad \subset\quad X_2
\]
has codimension at least 2.

Let $Z = \pi_2(\widetilde{Z})$ denote the projection of $\widetilde{Z}$ to $X_1$.
This closed subset has codimension at least 1 in $X_1$ and coincides with the ``parallel locus'' of the generators of degree at most $d-4$.
More explicitly, each generator of 
\[
    M_V = K'/eS, \qquad e = \begin{bmatrix}
        0\\ a_0\\ a_1\\ a_2
    \end{bmatrix},
\]
is a class represented as a four-tuple in $K'$, with entries in the homogeneous coordinate ring $S$ of $X_1$.
Hence, if $\sigma_1,\dots, \sigma_k \in K'$ ($k\geq 2$) are representatives for minimal homogeneous generator classes $\overline{\sigma_i}$ of $M_V$ in degree at most $d-4$, then $Z$ is the common zero locus of
\[
    \overline{\sigma_i}\wedge\overline{\sigma_j} = 0 \quad \text{in } \bigwedge^2 M_V \qquad  \text{for } 1 \leq i < j \leq k.
\]
This can be computed as the vanishing of the $3 \times 3$ minors of the matrix $[\sigma_1,\dots, \sigma_k, e]$.

The foliation on $Z$ is defined by the restriction $\O(-\xi_1)\big|_Z$. 
Assuming $Z$ dominates $\PP^2$, the associated multifoliation on $\PP^2$ can be explicitly described as follows: at a point $[x_0:x_1:x_2]\in\PP^2$, the direction corresponding to $a_0\frac{\partial}{\partial x_0}+a_1\frac{\partial}{\partial x_1}+a_2\frac{\partial}{\partial x_2}$ is permitted if and only if $([x_0:x_1:x_2], [a_0:a_1:a_2]) \in Z$.

To determine the leaves of the multifoliation on $\PP^2$, we must find the associated 1-form that vanishes on the distribution defined by $\O(-\xi_1)\big|_Z$. 
Since $T_{\PP^2}(-\log C) \subset T_{\PP^2}$, we have the inclusion $\O(-\xi_1) = \O(1,-1) \hookrightarrow \pi_1^* T_{\PP^2}$ and thus the distribution is defined by a global section in $H^0(Z,\pi_1^* T_{\PP^2} \otimes\O(-1,1)\big|_Z)$. 
Utilizing the canonical isomorphism $T_{\PP^2} \cong \Omega_{\PP^2}^1\otimes\O_{\PP^2}(3)$, we find that the multifoliation on $\PP^2$ is defined by a 1-form $\omega\in H^0(Z,\pi_1^*\Omega_{\PP^2}^1\otimes\O(2,1)\big|_Z)$.
In coordinates, we can write this as $\omega = A_0dx_0 + A_1dx_1 + A_2dx_2$, where the coefficients $A_i\in H^0(Z,\O(1,1)\big|_Z)$.
For $\omega$ to be a well-defined 1-form pulled back from the base and give rise to the multifoliation associated with $\O(-\xi_1)\big|_Z$, it must vanish along the two vector fields $\sum x_i\frac{\partial}{\partial x_i}$ and $\sum a_i\frac{\partial}{\partial x_i}$. 
In other words, $\omega$ must kill tangent vectors represented by vectors in $\C^3$ spanned by $\mathbf{x}=(x_0,x_1,x_2)$ and $\mathbf{a} = (a_0,a_1,a_2)$.
This uniquely determines $\omega$ up to scaling as 
\[
    \omega = (\mathbf{x} \times \mathbf{a}) \cdot \mathbf{dx}  = (x_1a_2 - x_2a_1)dx_0 + (x_2a_0-x_0a_2)dx_1 + (x_0a_1-x_1a_0)dx_2.
\]

Throughout the remainder of this subsection, we assume that $Z$ is a surface of pure codimension 1 in $X_1$ that dominates $\PP^2$, as is the case for all the examples considered here.
Since $Z$ is a Cartier divisor, it is defined by a global section in $H^0(X_1,\O(p,q))\cong S_{(p,q)}$ and thus represented by a bihomogeneous element of $S$. 
Let $h\in\C[x_0,x_1,x_2,a_0,a_1,a_2]$ denote a lift of this element to the homogeneous coordinate ring of $\PP^2\times\PP^2$, such that the ideal of $Z$ is generated by $(g,h)$.

To find the leaves of the multifoliation on $\PP^2$ we consider each of the irreducible components $Z_i = V(g,h_i)$, where $h_i$ is a prime factor of $h$.
If the component $Z_i$ does not dominate $\PP^2$, then we check if the projection to $\PP^2$ is hyperbolic in $\PP^2\setminus C$ by counting the intersection points with $C$.
If the component $Z_i$ dominates $\PP^2$, we consider three special cases.
\begin{description}
    \item[Case 1] Assume there is a homogeneous rational function $q(\mathbf{x})$ of degree 0 such that $\nabla q \cdot \mathbf{a}$ is a multiple of $h_i$.
    Then, the foliation on $\PP^2$ corresponding to the component $Z_i$ is defined by the vanishing of the differential $dq$.
    Indeed, the differential $dq$ evaluated at a tangent vector represented by a combination $\alpha\mathbf{x} + \beta\mathbf{a}$ is
    \[
        \nabla q \cdot (\alpha\mathbf{x} + \beta\mathbf{a}) = 0
    \]
    given that $\nabla q\cdot \mathbf{x} = 0$ by Euler's identity, and $\nabla q\cdot \mathbf{a} = 0$ because $h_i = 0$.
    Therefore, writing $q = F/G$, the leaves of the foliation are the pencil of curves $sF + tG = 0$, $[s:t]\in\PP^1$.
    \item[Case 2] Assume there is a homogeneous polynomial $q(\mathbf{x})$, of the same degree as $f$, for which $\nabla q \cdot \mathbf{a}$ is a multiple of $h_i$.
    Then, the meromorphic 1-form $d(q/f)$ evaluated at a tangent vector represented by a combination $\alpha\mathbf{x} + \beta\mathbf{a}$ is
    \[
        d(q/f)(\alpha\mathbf{x} + \beta\mathbf{a}) =  \frac{(\nabla q \cdot (\alpha\mathbf{x} + \beta\mathbf{a}))f - (\nabla f \cdot (\alpha\mathbf{x} + \beta\mathbf{a}))q}{f^2}.
    \]
    Euler's identity cancels the $\alpha$-terms, leaving
    \[
        \frac{\beta((\nabla q \cdot\mathbf{a})f-(\nabla f \cdot \mathbf{a})q)}{f^2}.
    \]
    The last expression is zero because $\nabla q \cdot\mathbf{a}$ is a multiple of $h_i$ and $\nabla f \cdot \mathbf{a} = g$.
    Hence, the foliation on $\PP^2$ associated to the component $Z_i$ is defined by the vanishing of the differential $d(q/f)$ and the leaves form the pencil of curves $sq + tf = 0$ for $[s:t]\in\PP^1$.
    \item[Case 3] Assume $h_i$ can be written as $h_i = E(\mathbf{x}, \mathbf{x}\times\mathbf{a})$, where $E(x_0,x_1,x_2,y_0,y_1,y_2)$ is a bihomogeneous polynomial with degree $k$ in the $y_i$ coordinates.
    Then we can find the leaves of the $k$-web associated to $Z_i$ by replacing $a_i$ with $dx_i$ and integrating the differential equation $E(\mathbf{x}, \mathbf{x}\times d\mathbf{x}) = 0$.
    This is because if we evaluate the symmetric differential $E(\mathbf{x}, \mathbf{x}\times d\mathbf{x})$ at a tangent vector represented by a linear combination $\alpha \mathbf{x} + \beta \mathbf{a}$, we get
    \[
        E(\mathbf{x}, \mathbf{x}\times (\alpha \mathbf{x} + \beta \mathbf{a})) = E(\mathbf{x},\beta \mathbf{x}\times \mathbf{a}) = \beta^kE(\mathbf{x}, \mathbf{x}\times \mathbf{a}) = 0.
    \]
\end{description}

\begin{example}\label{example:case 3(1)}
    Let 
    \[
        f = x_0^{2n} + x_0^nx_1^n + x_1^{2n} + x_1x_2^{2n-1}, \qquad n\geq 7.
    \]
    Note that $d = 2n$.
    The module $M_V$ has at least three minimal generators of degree at most $d-4$; two in degree $(n+2,1)$ and one in degree $(n+2,2)$.
    Formulas for these generators are written in \Cref{sec:Appendix}.
    Their parallel locus is
    \[
        (x_2)(x_1a_0-x_0a_1)(x_2a_1+(2n-1)x_1a_2) = 0.
    \]
    We say that $M_V$ has \textit{at least three} low-degree generators because the general formulas for these three generators were obtained from computations and verified to give independent sections for all $d = 2n\geq 16$.
    We do not rule out the existence of other minimal generators that may appear in sufficiently high degree. 
    
    If $n=7$ $(d=14)$, the two sections in degree $(n+2,1) = (9,1)$ are minimal homogeneous generators for $M_V$, whereas the section of degree $(n+2,2) = (9,2)$ is no longer a minimal generator.
    The parallel locus of the two minimal generators is
    \[
        (x_1)(x_2)(x_1a_0-x_0a_1)(x_2a_1+13x_1a_2) = 0.
    \]
    The only difference between the parallel loci in the cases $n=7$ and $n \geq 8$ is the presence of the component $x_1 = 0$ in the $n=7$ case. 
    However, we will see that this line is a leaf of the foliation on $\PP^2$ associated to the component $x_1a_0-x_0a_1 = 0$ and hence the description of the exceptional locus is the same for $n=7$ as for $n \geq 8$. 
    
    We restrict to each component separately to describe in detail how to find the corresponding leaves in $\PP^2\setminus C$.
    \begin{itemize}
        \item Let $x_2 = 0$. 
        This line intersects the curve $C$ at $2n$ distinct points and so it is hyperbolic in $\PP^2\setminus C$.
        \item Let $x_1a_0 - x_0a_1 = 0$. 
        In this case, we choose $q = \dfrac{x_1}{x_0}$ so that
        \[
            \nabla q\cdot \mathbf{a} = -\frac{1}{x_0^2}(x_1a_0 - x_0a_1) = 0.
        \]
        Hence, the leaves of the foliation on $\PP^2$ associated to this component are defined by the meromorphic 1-form $d(x_1/x_0)$ and thus form the pencil of lines $sx_0-tx_1 = 0$, $[s:t]\in\PP^1$.
        If $t = 0$ we get the line $x_0 = 0$ which intersects $C$ at the $2n$ points given by $x_1^{2n}+x_1x_2^{2n-1} = 0$.
        If $t\neq 0$ then the lines in the pencil are $x_1 = \lambda x_0$ with slope $\lambda = s/t$.
        Intersecting with $C$ we get the equation
        \[
            (\lambda^{2n}+ \lambda^n + 1)x_0^{2n} + \lambda x_0x_2^{2n-1} = 0,
        \]
        which has $2n$ distinct solutions whenever $\lambda\neq 0$ and $\lambda^{2n}+ \lambda^n + 1 \neq 0$.
        The line with $\lambda = 0$, i.e., $x_1 = 0$, intersects $C$ only at the base point $[0:0:1]$, and the lines with $\lambda^{2n} + \lambda^n + 1 = 0$ intersect $C$ only at the points $[1:\lambda:0]$ and $[0:0:1]$.
        Therefore, these $2n$ lines are part of the exceptional locus of $\PP^2\setminus C$.
        \item Let $x_2a_1+(2n-1)x_1a_2 = 0$. 
        We choose $q = x_1x_2^{2n-1}$ so that 
        \[
            \nabla q \cdot\mathbf{a} = x_2^{2n-2}(x_2a_1+(2n-1)x_1a_2) = 0.
        \]
        Then the leaves of the foliation on $\PP^2$ associated to this component are the pencil of curves
        \[
            s(x_0^{2n} + x_0^nx_1^n + x_1^{2n}) + tx_1x_2^{2n-1} = 0, \qquad [s:t]\in\PP^1.
        \]
        Note that the only two reducible members of the pencil are the ones with $[s:t]$ equal to $[1:0]$ and $[0:1]$; the rest are smooth and hence hyperbolic by having degree at least $14$.

        The curve $x_0^{2n} + x_0^nx_1^n + x_1^{2n} = 0$ is the union of the $2n$ lines $x_1 = \lambda x_0$, where $\lambda$ satisfies $\lambda^{2n} + \lambda^n + 1 = 0$; these are the same lines we found in the previous component.
        Finally, the curve $x_1x_2^{2n-1}$ is the union of the two lines $x_1 = 0$ and $x_2 = 0$, which we already analyzed.
    \end{itemize}

    If the module $M_V$ has additional minimal generators of low degree for $n$ beyond the range of our computations, the displayed generators would define a priori a larger parallel locus and consequently a larger exceptional set.
    However, the listed lines carry entire curves, so any additional generators cannot remove those lines.
    Therefore, these lines form the exceptional locus regardless of $M_V$ getting a new minimal generator for some $n\geq 7$.
    
    \noindent\textbf{Exceptional locus:} For $C = \{x_0^{2n} + x_0^nx_1^n + x_1^{2n} + x_1x_2^{2n-1} = 0\}$ and $d = 2n\geq 14$, the exceptional locus of $\PP^2\setminus C$ is the union of the $2n+1$ lines $\lambda x_0-x_1 = 0$ where $\lambda = 0$ or $\lambda^{2n}+ \lambda^n + 1 = 0$.
\end{example}

\begin{example}\label{example:case 3(2)}
    Let 
    \[
        f = x_0^{2n+1}  + x_0^{n}x_1^{n+1} + x_1^{2n+1} + x_1x_2^{2n} \qquad n \geq 7.
    \]
    If $n = 7$ $(d = 15)$ the module $M_V$ has two minimal homogeneous generators of total degree at most $d-4 = 11$: one in degree $(9,1)$ and the other one in degree $(10,1)$.
    Their parallel locus is the vanishing
    \[
        (x_1)(x_2)(x_1a_0-x_0a_1)(x_2a_1+14x_1a_2) = 0.
    \]
    For $n\geq 8$ ($d=2n+1\geq 17$) there are at least three minimal generators of total degree at most $d-4$, and these have degrees $(n+2,1)$, $(n+3,1)$, and $(n+3,2)$, respectively.
    Their parallel locus is cut out by
    \[
        (x_2)(x_1a_0-x_0a_1)(x_2a_1+2nx_1a_2) = 0.
    \]
    The calculation of the exceptional locus is analogous to the one done in \Cref{example:case 3(1)}.
    
    \noindent\textbf{Exceptional locus:} For $C = \{x_0^{2n+1}  + x_0^{n}x_1^{n+1} + x_1^{2n+1} + x_1x_2^{2n} = 0\}$ of degree $d = 2n+1$, the exceptional locus of $\PP^2\setminus C$ is the union of the $2n+2$ lines $\lambda x_0 - x_1 = 0$ where $\lambda = 0$ or $\lambda^{2n+1}+\lambda^{n+1} + 1 = 0$.
\end{example}

\begin{example}\label{example:case 3(3)}
    Let 
    \[
        f = x_0^{2n} + x_0^nx_1^n + x_1^{2n} + x_1x_2^{2n-1} + x_2^{2n}, \qquad n\geq 7.
    \]
    The module $M_V$ associated to these curves has at least two minimal homogeneous generators with total degree at most $d-4$. 
    Both generators occur in degree $(n+2,1)$ and their parallel locus is the vanishing 
    \[
        (x_2)\,((2n-1)x_1+2nx_2)\,(x_1a_0-x_0a_1)\,(x_2a_1+((2n-1)x_1+2nx_2)a_2) = 0.
    \]
    The lines $x_2 = 0$ and $(2n-1)x_1+2nx_2 = 0$ intersect $C$ at $2n$ distinct points and are thus hyperbolic in $\PP^2\setminus C$.
    As computed in \Cref{example:case 3(1)}, the leaves associated to the component $x_1a_0-x_0a_1 = 0$ form the pencil of lines $\lambda x_0 - x_1 = 0$.
    The lines with $\lambda^{2n} + \lambda^n + 1 = 0$ intersect $C$ at the points $[1:\lambda:0]$ (multiplicity $2n-1$) and $[1:\lambda:-\lambda]$ (multiplicity 1); all other lines, including $x_0 = 0$ ($\lambda = \infty$), intersect $C$ in at least $2n-1$ distinct points.

    The curves of the pencil $sf + t(x_1x_2^{2n-1} + x_2^{2n}) = 0$ form the leaves associated to the component $x_2a_1 + ((2n-1)x_1 + 2nx_2)a_2 = 0$.
    The base locus of this pencil consists of $4n$ distinct points. 
    The only two reducible members of the pencil are the ones corresponding to $[s:t] = [0:1]$ and $[1:-1]$, defined by the polynomials 
    \[
        x_2^{2n-1}(x_1 + x_2)
        \qquad \text{and} \qquad 
        x_0^{2n} + x_0^nx_1^n + x_1^{2n},
    \]
    respectively.
    All other members of the pencil are irreducible and hence hyperbolic after removing $4n$ points, i.e., in $\PP^2\setminus C$.
    The reduced zero locus $x_2^{2n-1}(x_1 + x_2) = 0$ consists of the lines $x_2 = 0$ and $x_1 + x_2 = 0$, which intersect $C$ at $2n$ distinct points, making them hyperbolic in $\PP^2\setminus C$.
    On the other hand, the zero locus $x_0^{2n} + x_0^nx_1^n + x_1^{2n} = 0$ consists of the $2n$ lines $\lambda x_0 -x_1 = 0$ with $\lambda^{2n} + \lambda^n + 1 = 0$, which are the exceptional lines found above.

    \noindent\textbf{Exceptional locus:} For $C = \{x_0^{2n} + x_0^nx_1^n + x_1^{2n} + x_1x_2^{2n-1} + x_2^{2n} = 0\}$ of degree $d=2n\geq 14$, the exceptional locus of $\PP^2\setminus C$ consists of the $2n$ lines $\lambda x_0-x_1 = 0$ with $\lambda^{2n}+\lambda^n + 1 = 0$.
\end{example}


\begin{example}\label{example:case 3(5)}
    Let
    \[
        f=x_0^{2n}+x_1^{2n}+x_2^{2n}+cx_0^nx_1^n,
        \qquad n\geq 6,\quad c\neq \pm2.
    \]
    The condition $c\neq\pm2$ guarantees that $C$ is smooth.
    The low-degree sections of $M_V$ depend slightly on the parameter $c$.
    If
    \[
        c\notin\left\{0, \pm2, \pm\frac{2(2n-1)}{n-1}\right\},
    \]
    then $M_V$ has at least two low-degree sections, both of degree $(n+1,1)$.
    For the special values $c=\pm\frac{2(2n-1)}{n-1}$, and $n\geq 7$, there are instead two low-degree sections, of degrees $(n,1)$ and $(n+2,1)$.
    Despite this change in degrees, their parallel locus is the same in both cases:
    \[
        (a_2)(x_2)(x_1a_0-x_0a_1)=0.
    \]
    Thus, for $n\geq7$ and $c\neq0,\pm2$, the calculation of the exceptional locus is the same.
    When $c=0$ (the Fermat curve), the two low-degree sections are no longer minimal generators and moreover vanish along a common component inside $X_2$.
    There is also one exceptional case: when $n=6$ and $c=\pm\frac{2(2n-1)}{n-1}=\pm 22/5$ there is only one low-degree section, of degree $(6,1)$, and we exclude these two values from the discussion below.
    
    Assume hereafter that $c\neq0,\pm2$, and, when $n=6$, that $c\neq\pm22/5$.
    The parallel locus of the relevant low-degree sections is
    \[
        (a_2)(x_2)(x_1a_0-x_0a_1)=0.
    \]
    First, let $a_2 = 0$.
    We choose $q = x_2^{2n}$.
    Then $\nabla q \cdot \mathbf{a} = 2nx_2^{2n-1}a_2 = 0$.
    Hence, the leaves of the foliation on $\PP^2$ associated to this component form the pencil
    \[
        s(x_0^{2n} + x_1^{2n} + c x_0^nx_1^n) + tx_2^{2n} = 0, \qquad [s:t]\in\PP^1.
    \]
    The curves $x_0^{2n} + x_1^{2n} + c x_0^nx_1^n = 0$ and $x_2^{2n} = 0$ are the only reducible  members of the pencil, and all other members are smooth and hence hyperbolic.
    The curve $x_0^{2n} + x_1^{2n} + c x_0^nx_1^n = 0$ consists of the $2n$ lines $x_1 = \lambda x_0$ where $\lambda$ satisfies $\lambda^{2n} + c\lambda^n + 1 = 0$.
    These lines intersect $C$ only at the point $[1:\lambda:0]$ and hence form part of the exceptional locus of $\PP^2\setminus C$.
    The other reducible member corresponds to the line $x_2 = 0$, which intersects $C$ at $2n$ distinct points and is therefore hyperbolic in $\PP^2\setminus C$.
    
    As we have seen for the component defined by $x_1a_0-x_0a_1 = 0$, the leaves on $\PP^2$ are the pencil of lines $x_1 = \lambda x_0$.
    The only lines in this pencil intersecting $C$ in at most 2 points are the same ones found above. 
    
    \noindent\textbf{Exceptional locus:} For $C = \{x_0^{2n} + x_1^{2n} + x_2^{2n} + c x_0^nx_1^n = 0\}$ with either $n\geq7$ and $c\neq0,\pm2$, or $n=6$ and $c\notin\{0,\pm2,\pm\frac{22}{5}\}$, the exceptional locus of $\PP^2\setminus C$ consists of the $2n$ lines $\lambda x_0-x_1 = 0$ with $\lambda^{2n}+c\lambda^n + 1 = 0$.    
\end{example}

\begin{example}
    Let 
    \[ 
        f = x_0^{14}+2x_0^7x_1^7+x_1^{14}-x_1^{13}x_2+x_2^{14}.
    \]
    The module $M_V$ has three minimal homogeneous generators of total degree at most $d-4 = 10$.
    These three sections have degree $(9,1)$ and their parallel locus is 
    \[
        (a_2)\,(x_2)\,\big(7x_2(14x_1 - 13x_2)(x_1a_0 - x_0a_1)^2 + 13x_0^2(x_1a_2 - x_2a_1)^2\big) = 0.
    \]
    Let $a_2 = 0$.
    We choose $q = x_2^{14}$.
    Then $\nabla q \cdot\mathbf{a} = 0$ and thus the leaves in $\PP^2$ form the pencil $sf + tq = 0$, which can be written as
    \[
        s((x_0^7 + x_1^7)^2 - x_1^{13}x_2) + tx_2^{14} = 0,\qquad [s:t]\in\PP^1.
    \]
    The base locus of this pencil is the set of 7 points $[\lambda:1:0]$, where $\lambda$ satisfies $\lambda^7 + 1 = 0$.
    Except for $x_2^{14} = 0$, every curve in the pencil is irreducible.
    As the pencil includes $C$, these leaves are all hyperbolic after removing the base points and thus hyperbolic in $\PP^2\setminus C$.
    Since the line $x_2 = 0$ also intersects $C$ at those 7 points, it is hyperbolic in $\PP^2\setminus C$.
    
    Now let us restrict to the component $7x_2(14x_1 - 13x_2)(x_1a_0 - x_0a_1)^2 + 13x_0^2(x_1a_2 - x_2a_1)^2 = 0$.
    This component has degree 2 in the $a$ variables, and therefore it defines a 2-web in $\PP^2$.
    As the polynomial defining the component is of the type described in Case 3 above, we find the leaves by integrating the symmetric differential form
    \[
        7x_2(14x_1 - 13x_2)(x_1dx_0 - x_0dx_1)^2 + 13x_0^2(x_1dx_2 - x_2dx_1)^2 = 0.
    \]
    Passing to affine coordinates $u = x_0/x_1$ and $v = x_2/x_1$, we obtain the separable equation
    \[
        7v(14 - 13v)(du)^2 + 13u^2(dv)^2 = 0 \implies \frac{13}{7} \frac{(dv)^2}{v(14 - 13v)} = -\frac{(du)^2}{u^2}.
    \]
    Taking square roots and integrating we obtain the general solution
    \[
        \frac{1}{\sqrt{7}} \arcsin\left(\frac{13v - 7}{7}\right) = \pm i \ln(u) + K.
    \] 
    Checking the line at infinity $x_1 = 0$ as well as the homogenization of the singular solutions $u = 0$, $v = 0$, and $14-13v = 0$, all these algebraic leaves intersect $C$ in at least 7 points, so they are hyperbolic in $\PP^2\setminus C$.
    
    In conclusion, $\PP^2\setminus C$ is hyperbolic for $C = \{x_0^{14}+2x_0^7x_1^7+x_1^{14}-x_1^{13}x_2+x_2^{14}=0\}$.
\end{example}

\subsection{Inconclusive Cases}
The following table collects examples of curves for which the module $M_V$ has only one minimal homogeneous generator of total degree at most $d-4$, but its degree does not satisfy the inequality in \eqref{ineq:module_too_negative_section}.
We give the range for the degree $d$ of the curve, where the lower bound is the first degree for which the section starts to appear.
The upper bound is the highest degree for which we confirmed the property about $M_V$ via computation. 
As a consequence, even in the degree range $d\geq 15$, we cannot conclude {\GGL} for the complement of these curves by means of \Cref{Module test} (equiv. \Cref{test theorem}).

\begin{example}\hfill
\begin{center}
    \renewcommand{\arraystretch}{1.4} 
    \begin{tabular}{@{} l c c @{}}
        \toprule
        Polynomial defining $C\subset\PP^2$ & Degree of generator & Range computed \\
        \midrule
        $x_0^d + x_1^d + x_2^d$ & $(2,1)$ &  $7 \leq d \leq 20$ \\
        $x_0^d + x_2^d - x_1^{d-1}x_2$ & $(3,1)$ & $8 \leq d \leq 20$ \\
        $x_0^d + x_1^d + x_2^d + x_0^{d-1}x_1$ & $(3,1)$ & $8 \leq d \leq 20$ \\
        $x_0^d + x_1(x_0^{d-1} + x_1^{d-1} + x_2^{d-1})$ & $(4,1)$ & $9 \leq d \leq 20$ \\
        $x_0^d + x_1^d + x_2^{d-2}(x_0^2 + x_1^2 + x_2^2)$ & $(4,1)$ & $9 \leq d \leq 20$ \\
        $x_0^{d-1}x_1 + x_1^{d-1}x_2 + x_2^{d-1}x_0$ & $(5,1)$ & $10 \leq d \leq 20$ \\
        $x_0^d + x_1^d + x_2^{d-2}(x_0^2 + x_0x_1 + x_0x_2 + x_1^2 + x_1x_2 + x_2^2)$ & $(4,1)$ & $d = 15, 16$\\
        \bottomrule
    \end{tabular}
\end{center}
\end{example}


\input{appendix}

\newpage
\printbibliography

\end{document}

%% file: appendix.tex
\section{Appendix: Formulas for the Generators of \Cref{example:case 3(1)}}\label{sec:Appendix}
\noindent\textbf{Generator 1} of bidegree $(n+2,1)$.
\begin{align*}
G_{11} & = (2n-1)x_1^2x_2\big[(7n-3)x_0^n + 2nx_1^n\big] \\[5pt]
G_{21} & = 2n(2n-1)^2x_0 x_1^n x_2 a_1^2 \\[5pt]
G_{31} & = -(2n-1)x_1x_2 \begin{aligned}[t] 
    \big[&n(n-1) x_0^{n-2} x_1^2  a_0^2\\
    & - (16n^2 - 15n + 3) x_0^{n-1} x_1  a_0a_1\\
    & - ((15n^2 - 14n + 3)x_0^n + 2nx_1^n )  a_1^2 \big]
    \end{aligned}\\[5pt]
G_{41} &= \begin{aligned}[t]  &n(n - 1)x_0^{n - 2}x_1^2x_2^2 a_0^2 \\
&- 2n(n - 1)x_0^{n - 1}x_1x_2^2 a_0a_1 \\
&+ x_2^2\big((15n^2 - 14n + 3)x_0^n + 2n(2n - 1)x_1^n\big) a_1^2 \\
&+ (2n - 1)^2(7n - 3)x_0^{n - 1}x_1^2x_2 a_0a_2 \\
&- 2(2n-1)x_1x_2\big((7n - 3)x_0^n - 4n(n - 1)x_1^n\big)  a_1 a_2 \\
&+ (n-1)(2n-1)x_1^2 \big(-2(7n - 3)x_0^n - 4nx_1^n\big)  a_2^2
\end{aligned}
\end{align*}

\noindent\textbf{Generator 2} of bidegree $(n+2,1)$.
\begin{align*}
G_{12} &= x_0x_1x_2\big[2n x_0^n + (7n-3) x_1^n\big] \\[5pt]
G_{22} &= -n(n-1) x_0^2 x_1^{n-1} x_2 a_1^2 \\[5pt]
G_{32} &= x_1x_2 \big[\begin{aligned}[t]
    &2n(2n-1) x_0^{n-1} x_1 a_0^2 \\
    &-(2n(2n-1) x_0^n - (5n-3)(3n-1) x_1^n) a_0 a_1 \\
    &-(16n^2 - 15n + 3) x_0 x_1^{n-1}a_1^2 \big]
\end{aligned}\\[5pt]
G_{42} &= - 2n x_0^{n-1} x_1 x_2^2 a_0^2 \\
&+ 4n x_0^n x_2^2 a_0 a_1 \\
&+ (7n-3) x_0 x_1^{n-1} x_2^2 a_1^2 \\
&+ x_2x_1\big[2n(2n-1) x_0^n + (5n-3)(3n-1) x_1^n\big] a_0 a_2 \\
&- 2x_0x_2\big[2n x_0^n + (n^2 + 6n - 3) x_1^n\big] a_1 a_2 \\
&- 2(n-1)x_0x_1\big[2n x_0^n + (7n-3)x_1^n\big] a_2^2
\end{align*}

\noindent\textbf{Generator 3} of bidegree $(n+2,2)$.
{\allowdisplaybreaks
\begin{align*}
G_{13} &= x_0\big[ \begin{aligned}[t]
    & 4n(2n-1)x_1x_2\bigl((7n - 3)x_0^n + 2nx_1^n \bigr)a_0 \\
    & - (3n-1)(5n-3)x_2^2\bigl(nx_0^n + (7n - 3)x_1^n\bigr)a_1  \\
    & + 6(2n - 1)(3n - 1)(5n - 3)x_1^{n+1}x_2a_2\bigr]
\end{aligned} \\[5pt]
G_{23} &= n(3n-1)(5n-3)x_0^2 x_1^{n-2}x_2\bigl[(n - 1)x_2a_1 - 2(2n - 1)x_1a_2 \bigr] a_1^2\\[5pt]
G_{33} &= \begin{aligned}[t]
    & - 4n^2(n-1)(2n - 1)x_0^{n-2}x_1^2x_2^2 a_0^3 \\
    & + 2n(2n - 1)(17n^2 - 16n + 3)x_0^{n-1}x_1x_2^2 a_0^2a_1 \\
    & - x_2^2\bigl[2n(2n - 1)(3n - 1)(5n - 3)x_0^n + (193n^4 - 388n^3 + 278n^2 - 84n + 9)x_1^n \bigr] a_0a_1^2 \\
    & + (208n^4 - 417n^3 + 295n^2 - 87n + 9)x_0x_1^{n-1}x_2^2 a_1^3 \\
    & + 4n(2n - 1)(3n - 1)(5n - 3)x_0^{n-1}x_1^2x_2 a_0^2a_2 \\
    & - 2(2n-1)(3n-1)(5n-3)x_1x_2 \bigl[ 4nx_0^n + (7n - 3)x_1^n\bigr] a_0a_1a_2 \\
    & + 2(2n-1)(3n-1)(5n-3)x_0x_2 \bigl[ 2nx_0^n + (3 - 8n)x_1^n\bigr] a_1^2a_2
\end{aligned}\\[5pt]
G_{43} &= \begin{aligned}[t]
&4n^2(n - 1)x_0^{n-2}x_1x_2^3a_0^3\\
&+2n(11n^2 - 10n + 3)x_0^{n-1}x_2^3a_0^2a_1\\
&+8n^2(2n - 1)x_1^{n-1}x_2^3a_0a_1^2\\
&-(3n - 1)(5n - 3)(7n - 3)x_0x_1^{n-2}x_2^3a_1^3\\
&+4n(n - 1)(28n^2 - 27n + 6)x_0^{n-1}x_1x_2^2a_0^2a_2\\
&-x_2^2\bigl[2n(30n^3 - 47n^2 + 24n - 3)x_0^n + (161n^4 - 324n^3 + 254n^2 - 84n + 9)x_1^n\bigr]a_0a_1a_2\\
&-2(n^4 - 182n^3 + 259n^2 - 120n + 18)x_0x_1^{n-1}x_2^2a_1^2a_2\\
&-2(2n-1)x_1x_2\bigl[4n(n - 1)(7n - 3)x_0^n - (97n^3 - 135n^2 + 63n - 9)x_1^n\bigr]a_0a_2^2 \\
&+2(3n-1)(5n-3)x_0x_2\bigl[2n(n - 1)x_0^n + (3n^2 - 20n + 9)x_1^n\bigr]a_1a_2^2 \\
&-12(n - 1)(2n - 1)(3n - 1)(5n - 3)x_0x_1^{n+1}a_2^3
\end{aligned}
\end{align*}
}